\documentclass{amsart}
\usepackage{latexsym,amsxtra,amscd,ifthen,amsmath,color, multicol,mathdots,fancyvrb,url, bbm}
\usepackage{amsfonts}
\usepackage{verbatim}
\usepackage{amsmath}
\usepackage{amsthm}
\usepackage{amssymb, mathtools}
\usepackage{graphicx} 
\usepackage{caption} 
\usepackage{subcaption}
\usepackage[notcite,notref,final]{showkeys}
 \usepackage[all,cmtip]{xy}

\numberwithin{equation}{section}

\theoremstyle{plain}
\newtheorem{theorem}{Theorem}[section]
\newtheorem{lemma}[theorem]{Lemma}

\newtheorem{proposition}[theorem]{Proposition}

\newtheorem{corollary}[theorem]{Corollary}

\newtheorem{definitiontheorem}[theorem]{Definition-Theorem}

\theoremstyle{definition}
\newtheorem{definition}[theorem]{Definition}
\newtheorem{example}[theorem]{Example}

\newtheorem{notation}[theorem]{Notation}

\newtheorem{remark}[theorem]{Remark}

\newcommand{\stkout}[1]{\ifmmode\text{\sout{\ensuremath{#1}}}\else\sout{#1}\fi}

\makeatletter              
\let\c@equation\c@theorem  
\makeatother

\newcommand{\C}{\mathcal{C}}
\newcommand{\M}{\mathcal{M}}
\newcommand{\id}{\textnormal{id}}
\newcommand{\kk}{\Bbbk}
\newcommand{\vecgw}{{\sf Vec}_G^\omega}
\newcommand{\unit}{\mathbbm{1}}
\newcommand{\Hom}{\textnormal{Hom}}

\definecolor{forest}{rgb}{0.0, 0.5, 0.0}
\definecolor{brandeisblue}{rgb}{0.0, 0.2, 1.0}
\definecolor{blue-violet}{rgb}{0.54, 0.17, 0.89}
\definecolor{cyan(process)}{rgb}{0.0, 0.6, 1.0}

\begin{document}

\title[Algebraic structures in group-theoretical fusion categories]{Algebraic structures in group-theoretical fusion categories}
\author[Morales et al.]{Yiby Morales, Monique M\"{u}ller, Julia Plavnik, Ana Ros Camacho, Angela Tabiri, and Chelsea Walton$^*$}

\address{Morales: Departamento de Matem\'{a}ticas, Universidad de los Andes, Cra. 1 \#18a 12, Bogot\'{a}, Colombia}
\email{yk.morales964@uniandes.edu.co}

\address{M\"{u}ller: Departamento de Matem\'atica e Estat\'istica, Universidade Federal de S\~ao Jo\~ao del-Rei, Pra\c ca Frei Orlando, 170, Centro, S\~ao Jo\~ao del-Rei, Minas Gerais, Brazil, CEP: 36307-352 }
\email{monique@ufsj.edu.br}

\address{Plavnik: Department of Mathematics, Indiana University, Bloomington, IN 47405 USA}
\email{jplavnik@iu.edu}

\address{Ros Camacho: School of Mathematics, Cardiff University, Senghennydd Road, CF24 4AG Cardiff, Wales}
\email{roscamachoa@cardiff.ac.uk}

\address{Tabiri: African Institute for Mathematical Sciences Ghana, Summerhill Estates, GPS: GK- 0647-1372, Accra, Ghana}
\email{angela@aims.edu.gh}

\address{Walton: Department of Mathematics, Rice University, Houston, TX 77005 USA}
\email{notlaw@rice.edu}

\bibliographystyle{abbrv}

\begin{abstract}
It was shown by Ostrik (2003) and Natale (2017) that a collection of twisted group algebras in a pointed fusion category serve as explicit Morita equivalence class representatives of indecomposable, separable algebras in such categories. We generalize this result by constructing explicit Morita equivalence class representatives of indecomposable, separable algebras in group-theoretical fusion categories. This is achieved by providing the free functor $\Phi$ from  fusion category to a category of bimodules in the original category with a (Frobenius) monoidal structure. Our algebras of interest are then constructed as the image of twisted group algebras under $\Phi$. We also show that twisted group algebras admit the structure of Frobenius algebras in a pointed fusion category, and as a consequence, our algebras are Frobenius algebras in a group-theoretical fusion category. They also enjoy several good algebraic properties. 
\end{abstract}

\keywords{free functor, Frobenius algebra, Frobenius monoidal functor, group-theoretical fusion category, Morita equivalence, pointed fusion category, separable algebra\\
\indent * Corresponding author}

\maketitle



\allowdisplaybreaks

\section{Introduction}

The goal of this work is to construct explicit algebras that represent Morita equivalence classes in group-theoretical fusion categories, and that possess good algebraic properties. Throughout, we assume that $\kk$ is an algebraically closed field of characteristic 0.

A group-theoretical fusion category is a certain kind of semisimple monoidal category whose construction depends on group-theoretic data, and we will restrict our attention to such categories below. But for now let us discuss the prevalence of Morita equivalence of algebras  in general. Recall that two rings are said to be {\it Morita equivalent} if their categories of modules are equivalent as categories. Many nice properties are preserved under such an equivalence including the Noetherian, (semi)simple, (semi)hereditary, and (semi)prime conditions \cite[Chapter~7]{Lam}. The notion of Morita equivalence has been upgraded for algebras of various types, and is used in several areas  including $C^*$-algebras \cite{BGR}, Poisson geometry \cite{Xu}, and various subfields of physics \cite{DN, FRS, SW}. In all of these cases, one is studying the Morita equivalence of algebras (or, of algebra objects) in a fixed monoidal category.

 Two algebras in $\C$ are said to be {\it Morita equivalent} if their categories of (right) modules in $\C$ are equivalent as (left) $\C$-module categories.  For a {\it fusion category} $\C$ the main result of \cite{OstrikTG2003} states that any $\C$-module category $\mathcal{M}$ is equivalent to the category of modules over some algebra $A$ in $\C$, and the algebra $A$ used in the proof of this result is an {\it internal End} of any nonzero object of $\mathcal{M}$ (see \cite[Section~3.2]{OstrikTG2003}). It is also shown that this internal End, $A$, can be taken to be  connected [Definition~\ref{def:algprop}], but no other good algebraic properties of $A$ are established nor is the construction of $A$ explicit. In our work, we restrict our attention to certain types of fusion categories that depend on group-theoretic data and we produce Morita equivalence representatives of algebras in these categories that depend explicitly on this group-theoretic data. An important class of fusion categories are pointed fusion categories, that is,  the categories $\vecgw$, with $G$ a finite group and $\omega \in H^3(G, \kk^\times)$, consisting of  $G$-graded $\kk$-vector spaces with associativity constraint $\omega$. The simple objects of $\vecgw$ are 1-dimensional $\kk$-vector spaces, denoted by $\{\delta_g\}_{g \in G}$, with $G$-grading $(\delta_g)_x = \delta_{g,x}\; \kk$, for $g,x \in G$.

\begin{definition} Let $L$ be a subgroup of $G$ so that $\omega|_{L^{\times 3}}$ is trivial, and take a 2-cochain $\psi \in C^2(L, \kk^\times)$ so that $d\psi = \omega|_{L^{\times 3}}$. The {\it twisted group algebra} $A(L,\psi)$ in $\vecgw$ is $\bigoplus_{g \in L} \delta_g$, with multiplication given by $\delta_g \otimes \delta_{g'} \mapsto \psi(g,g') \delta_{gg'}.$
\end{definition}

We have the following construction and result due to work of V. Ostrik and work of S. Natale.

\begin{theorem}\cite[Example~2.1]{OstrikIMRN2003} \cite[Example~9.7.2]{EGNO} \cite{Natale2016}
\label{thm:OstNat-intro}
A collection of twisted group algebras $A(L,\psi)$ serve as Morita equivalence class representatives of  indecomposable, separable algebras in the pointed fusion category $\vecgw$. \qed
\end{theorem}

The first of our results is that we establish a Frobenius algebra structure on the twisted group algebras and study related algebraic properties. See Definition~\ref{def:algprop} for a description of  some properties mentioned for algebras in fusion categories.

\begin{proposition}[Propositions~\ref{prop:ALpsi} and~\ref{prop:Alpsi props}] \label{prop:ALpsi-intro}
The twisted group algebras $A(L,\psi)$ admit the structure of a Frobenius algebra in $\vecgw$. They are also  connected and special.  \qed
\end{proposition}

Now we turn our attention to {\it group-theoretical fusion categories}. Introduced by P. Etingof, D. Nikshych, and V. Ostrik in \cite[Section~8.8]{ENO}, these are the categories $\mathcal{C}(G,\omega,K,\beta)$ consisting of $A(K,\beta)$-bimodules in $\vecgw$, for $G$ and $\omega$ as above, and with $K$ a subgroup of $G$ so that $\omega|_{K^{\times 3}}$ is trivial, and $\beta \in C^2(K, \kk^\times)$ so that $d\beta = \omega|_{K^{\times 3}}$. (See also \cite[Section~9.7]{EGNO}.) Group-theoretical fusion categories are a vital part of  the classification program of  fusion categories (see, e.g., \cite[Theorem~9.2]{ENO-wgtfc} and \cite[Section~9.13]{EGNO}), and due to their explicit construction, they also serve as a go-to testing ground for results about fusion categories (see, e.g.,  \cite[Section~5]{EKW}, \cite[Corollary~4.4]{ERW}, \cite{GP}, \cite[Section~4]{Gelaki}, \cite{Natale-FS}, \cite{OstrikIMRN2003}).

Towards our goal of constructing nice Morita equivalence class representatives of algebras in group-theoretical fusion categories, we start  by considering in a more general setting the free functor $\Phi$ from a fusion category $\C$ to a category of bimodules in $\C$, and  endow this functor with further structure (see Definition~\ref{def:monfunc}).

\begin{theorem}[Theorem~\ref{thm:Phi}] \label{thm:Phi-intro}
Let $\C$ be a fusion category, and let $A$ be a special Frobenius algebra in $\C$. Let ${}_A \C_A$ denote the monoidal category of $A$-bimodules in~$\C$. Then, the  free  functor $\Phi: \C \to {}_A \C_A$  is Frobenius monoidal. \qed
\end{theorem}

The result above enables us to define  algebraic structures that will fulfill our~goal.

\begin{definitiontheorem}[Definition~\ref{def:tha}, Theorem~\ref{thm:AKL}] Using the functor $\Phi$ above in the case when $\C = \vecgw$ and $A = A(K,\beta)$, we define the \textnormal{twisted Hecke algebra} $A^{K,\beta}(L,\psi)$ to be the algebra $\Phi(A(L,\psi))$ in $\C(G,\omega, K, \beta)$. It admits the structure of a  Frobenius algebra in $\C(G,\omega, K, \beta)$. \qed
\end{definitiontheorem}

The terminology is due to the fact that simple objects of group-theoretical fusion categories $\mathcal{C}(G,\omega,K,\beta)$ are in part parameterized by $K$-double cosets in $G$ (see Lemma~\ref{lem:simple-gt}), and the multiplication of $A^{K,\beta}(L,\psi)$ is twisted by cocycles $\beta$ and $\psi$. Twisted Hecke algebras also enjoy several nice algebraic properties.

\begin{proposition}[Proposition~\ref{prop: AKL props}] \label{prop:tha-intro}
The twisted Hecke algebras $A^{K,\beta}(L,\psi)$ are indecomposable, separable algebras in $\mathcal{C}(G,\omega,K,\beta)$, and are special Frobenius. \qed
\end{proposition}

We provide a precise condition describing when twisted Hecke algebras are connected in Proposition~\ref{prop:non-connected}. 
Finally, our goal is achieved as follows.

\begin{theorem}[Theorem~\ref{thm:modtha-new}]
\label{thm:modtha-new-intro}
A collection of  twisted Hecke algebras $A^{K,\beta}(L,\psi)$ serve as  Morita equivalence class representatives of indecomposable, separable algebras in the group-theoretical fusion category $\mathcal{C}(G,\omega,K,\beta)$. \qed
\end{theorem}

An application of this result to P. Etingof, R. Kinser, and the last author's study of tensor algebras in group-theoretical fusion categories \cite{EKW} is discussed in Remark~\ref{rem:EKW} and Example~\ref{ex:EKW}.

Theorem~\ref{thm:modtha-new-intro} is achieved by introducing the notion of a {\it Morita preserving \linebreak monoidal  functor} [Theorem~\ref{thm:Moritapres}, Definition~\ref{def:Moritapres}] and by establishing the following general result.

\begin{theorem}[Theorem~\ref{thm:Morita-updown}] \label{thm:Morita-updown-intro}
Let $\C$ be a fusion category. Take a connected, special Frobenius algebra $A$ in $\C$, and take algebras $B$, $B'$ in~$\C$.  Consider the monoidal functor $\Phi$ from Theorem~\ref{thm:Phi-intro}. Then, $B$ and $B'$ are Morita equivalent as algebras in~$\C$ if and only if $\Phi(B)$ and $\Phi(B')$ are Morita equivalent as algebras in ${}_A \C_A$. \qed
\end{theorem}

Indeed, with Theorem~\ref{thm:OstNat-intro} (due to Ostrik and Natale) and Proposition~\ref{prop:ALpsi-intro}, Theorem~\ref{thm:Morita-updown-intro}  provides the crucial step in proving Theorem~\ref{thm:modtha-new-intro} by setting $\C = \vecgw$, $A = A(K,\beta)$, $B = A(L,\psi)$, $B' = A(L',\psi')$.

\medskip

Our paper is organized as follows. We provide background material on fusion categories, module categories, and algebraic structures within them in Section~\ref{sec:prelim}. In Section~\ref{sec:Phi}, we establish Theorem~\ref{thm:Phi-intro}, and in Section~\ref{sec:MoritaPhi}, we obtain Theorem~\ref{thm:Morita-updown-intro}. Then, Proposition~\ref{prop:ALpsi-intro} is obtained in Section~\ref{sec:pointed}. Proposition~\ref{prop:tha-intro} is proved in Section~\ref{sc: algebras in GT}. Finally, the main result, Theorem~\ref{thm:modtha-new-intro}, is verified in Section~\ref{sec:repgtfc} by combining the results above.


\section{Preliminaries on fusion categories} \label{sec:prelim}

In this  section, we provide background information and preliminary results on fusion categories. We consider the following terminology. 

\begin{definition}[$(\mathcal{C}, \otimes, \unit, \alpha, l, r)$] \cite[Sections~2.1, ~2.2, ~2.10 and ~4.1]{EGNO}
\begin{enumerate}
    \item 
A \textit{monoidal category} $\mathcal{C}$ consists of the following data: a category $\C$; a bifunctor $\otimes : \C \times \C\rightarrow \C$; an object $\unit \in \C$; a natural isomorphism  $$\alpha_{X,X',X''}: (X \otimes X')\otimes X'' \overset{\sim}{\to} X \otimes (X'\otimes X'')$$ for each $X,X',X'' \in \C$; natural isomorphisms $$l_X: \mathbbm{1} \otimes X \overset{\sim}{\to} X, \quad \quad r_X:X \otimes \mathbbm{1} \overset{\sim}{\to} X$$ for each $X \in \C$,
 such that the pentagon and triangle coherence conditions are satisfied \cite[(2.2),(2.10)]{EGNO}.
 \smallskip
    \item An object in a monoidal category $\C$ is called \emph{rigid} if it has left and right duals.  Namely, for each $X \in \C$, there exists objects $X^*$ and ${}^* \hspace{-.015in} X \in \C$ so that we have co/evaluation maps 
    \[
    \begin{array}{c}
    \smallskip
    {\sf ev}_X: X^* \otimes X \to \unit, \quad\quad {\sf coev}_X:  \unit \to X \otimes X^*,\\
    {\sf ev}'_X: X \otimes {}^* \hspace{-.015in} X \to \unit, \quad \quad {\sf coev}'_X:  \unit \to {}^* \hspace{-.015in} X \otimes X,
    \end{array}
    \]
    satisfying compatibility conditions \cite[(2.43)--(2.46)]{EGNO}.
    Further, a monoidal category $\mathcal{C}$ is called \emph{rigid} if each of its objects is rigid. 
    \smallskip
    \item  A  $\kk$-linear, abelian, semisimple,  finite, rigid monoidal category $\mathcal{C}$  is a {\it fusion category} over $\kk$ if {\sf End}$_{\mathcal{C}}(\unit) \cong \kk$. 
\end{enumerate}
\end{definition}

\begin{remark}
Note that the tensor product for fusion categories is exact in both factors \cite[Proposition~4.2.1]{EGNO}. 
\end{remark}

\begin{center}
{\it We assume that $\mathcal{C}$ is a fusion category over $\kk$ throughout this article, \\ unless otherwise specified.}
\end{center}

\subsection{Monoidal functors and module categories} \label{sec:monoidal}

\begin{definition}  \label{def:monfunc}
\cite[page~85]{Str} \cite{DP08} \cite[(6.46), (6.47)]{Sz}
Let $\C$, $\mathcal{D}$ be  monoidal categories.
\begin{enumerate}
    \item  A \textit{monoidal functor} $(F, F_{\ast,\ast}, F_0): \C \to \mathcal{D}$ consists of 
    a functor $F: \C \to \mathcal{D}$,
    a natural transformation $F_{X,X'}:F(X)\otimes_{\mathcal{D}} F(X') \to F(X \otimes_{\C} X')$ for all $X,X' \in \C$, and
     a morphism $F_0: \unit_{\mathcal{D}} \to F(\unit_{\C})$ in $\mathcal{D}$,
    that satisfy the following associativity and unitality constraints,  
    \[
    \begin{array}{ll}
    \smallskip
    &F_{X,X'\otimes_\C X''}\;(\id_{F(X)} \otimes_{\mathcal{D}} F_{X',X''})\;\alpha_{F(X),F(X'),F(X'')}\\
    &\hspace{1.3in} = F(\alpha_{X,X',X''})\;F_{X \otimes_\C X',X''}\;(F_{X,X'}\otimes_{\mathcal{D}}\id_{F(X'')}),
    \end{array}
    \]
    
    \vspace{-.04in}
    
    \[
    \begin{array}{rl}
    \smallskip
    F(l_X)^{-1}\;l_{F(X)} &= F_{\unit_\C,X} \; (F_0 \otimes_{\mathcal{D}} \id_{F(X)}),\\
    F(r_X)^{-1}\;r_{F(X)} &= F_{X,\unit_\C} \; (\id_{F(X)} \otimes_{\mathcal{D}} F_0).
    \end{array}
    \]
    \smallskip
    \item  A \textit{comonoidal functor} $(F, F^{\ast,\ast}, F^0): \C \to \mathcal{D}$ consists of a  functor $F: \C \to \mathcal{D}$,  a  natural transformation $F^{X,X'}: F(X \otimes_{\C} X') \to F(X)\otimes_{\mathcal{D}} F(X')$ for all $X,X' \in \C$, and a  morphism $F^0: F(\unit_{\C}) \to  \unit_{\mathcal{D}}$ in $\mathcal{D}$, that satisfy the following coassociativity and counitality constraints, 
        \[
    \begin{array}{ll}
    \smallskip
    &\alpha^{-1}_{F(X),F(X'),F(X'')}\;(\id_{F(X)} \otimes_{\mathcal{D}} F^{X',X''})\;F^{X,X'\otimes_\C X''}\\
    &\hspace{1.3in} = (F^{X,X'}\otimes_{\mathcal{D}}\id_{F(X'')})\; F^{X \otimes_\C X',X''}\;F(\alpha^{-1}_{X,X',X''}),
    \end{array}
    \]
    
     \vspace{-.02in}
     
    \[
    \begin{array}{rl}
    \smallskip
    F(l_X) &= l_{F(X)}\; (F^0 \otimes_{\mathcal{D}} \id_{F(X)})\; F^{\unit_\C,X},\\
    F(r_X) &= r_{F(X)} \;(\id_{F(X)} \otimes_{\mathcal{D}} F^0)\; F^{X,\unit_\C}.
    \end{array}
    \]
    \smallskip
    \item  A \textit{Frobenius monoidal functor}  $(F, F_{\ast,\ast}, F_0, F^{\ast,\ast}, F^0)$ from $\C$ to $\mathcal{D}$ is a functor where $(F, F_{\ast,\ast}, F_0)$ is monoidal and $(F, F^{\ast,\ast}, F^0)$ is comonoidal, such that for all $X,X',X'' \in \C$:
    \[
    \begin{array}{l}
    \smallskip
    \hspace{.5in} (F_{X,X'} \otimes_{\mathcal{D}} \id_{F(X'')} )\alpha^{-1}_{F(X),F(X'),F(X'')}(\id_{F(X)} \otimes_{\mathcal{D}} F^{X',X''})\\ \medskip
    \hspace{2.5in} = F^{X\otimes_{\C}X',X''}F(\alpha^{-1}_{X,X',X''}) F_{X,X'\otimes_{\C}X''}, \\  \smallskip
    \hspace{.5in} (\id_{F(X)} \otimes_{\mathcal{D}} F_{X',X''}) \alpha_{F(X),F(X'),F(X'')} (F^{X,X'}\otimes_{\mathcal{D}} \id_{F(X'')})\\ 
    \hspace{2.5in} = F^{X,X'\otimes_{\C}X''}F(\alpha_{X,X',X''})F_{X\otimes_{\C}X',X''}.
    \end{array}
    \]
\end{enumerate}
\end{definition}

\smallskip

Here, `monoidal' means `lax monoidal' in other references.  Strong monoidal functors are monoidal functors where $F_{*,*}$ and $F_0$ are isomorphisms in $\mathcal{D}$, and we do not require this condition here.

\begin{definition}(see, e.g., \cite[Sections~7.1,~7.2]{EGNO}) \label{def:Cmodcat} 
Let $\C$  be a fusion  category.
\begin{enumerate}
\item A {\it left $\C$-module category} is a semisimple, $\kk$-linear, abelian category $\mathcal{M}$ equipped with 
a bifunctor $\otimes: \C \times \mathcal{M}\to \mathcal{M}$ bilinear on morphisms and exact,
 natural isomorphisms for associativity
    $$ \quad \quad \quad \quad  m_{X,Y,M}: (X \otimes Y) \otimes M \overset{\sim}{\to} X \otimes (Y \otimes M), \quad \forall X, Y \in \C, \; M \in \M$$
    satisfying the pentagon axiom, and
 for each $M \in \mathcal{M}$ a natural isomorphism $\unit \otimes M \overset{\sim}{\to} M$ satisfying the triangle axiom \cite[(7.2), (7.4)]{EGNO}.
 
{\it Right $\C$-module categories} are defined analogously. 

\smallskip

\item A module category $\mathcal{M}$ over $\C$ is {\it indecomposable} if it is nonzero and is not equivalent to a direct sum of two nontrivial module categories over $
\C$.

\smallskip

\item Let $\mathcal{M}$ and $\mathcal{N}$ be two left $\C$-module categories. A {\it (left) $\C$-module functor} from $\mathcal{M}$ to $\mathcal{N}$ is a functor $F: \mathcal{M} \to \mathcal{N}$ equipped with a natural isomorphism $s_{X,M}: F(X \otimes M) \overset{\sim}{\to} X \otimes F(M)$ for each $X \in \C$, $M \in \mathcal{M}$ satisfying the pentagon and triangle axioms \cite[(7.6), (7.7)]{EGNO}.
{\it Right $\C$-module functors} are defined analogously.

\smallskip

\item An {\it equivalence of $\C$-module categories} is a $\C$-module functor $(F,s)$ so that $F: \mathcal{M} \to \mathcal{N}$ is an equivalence of categories.
\end{enumerate}
\end{definition}


\subsection{Algebraic structures in fusion  categories}  \label{sec:alg-mon}

Now we recall the notion of an algebra, a coalgebra, and a Frobenius algebra in a fusion  category. For general information, see \cite[Section~2]{FRS-3}, \cite[Section~3]{OstrikTG2003}, \cite[Section~7.8]{EGNO}, and references within.

\begin{definition}[${\sf Alg}(\C)$, ${\sf Coalg}(\C)$, ${\sf FrobAlg}(\C)$] \label{def:algstr}  Let $\C$ be a  monoidal category.
\begin{enumerate}
    \item An \textit{algebra} in $\C$ is a triple $(A,m,u)$, with $A \in \C$, and  $m:A \otimes A \to A$ (multiplication), $u:\unit \to A$ (unit) being morphisms in $\C$,
    satisfying unitality and associativity constraints: 
    
    \vspace{-.15in}
    
    \begin{align*} 
      \hspace{.5in}  m (m \otimes \id) = m(\id \otimes m) \alpha_{A,A,A}, \quad \quad
       m (u \otimes \id) = l_A, \quad   \quad m(\id \otimes u) = r_A. 
    \end{align*}
    
    \smallskip
    
    \noindent A {\it morphism} of algebras $(A, m_A, u_A)$ to $(B, m_{B}, u_{B})$  is a map $f: A \to B$ in $\C$ so that $fm_A = m_{B}(f \otimes f)$ and $fu_A = u_{B}$.  Algebras in $\C$ and their morphisms  form a category, which we denote by ${\sf Alg}(\C)$.

    \medskip
    
    \item A \textit{coalgebra} in $\C$ is a triple $(C,\Delta,\varepsilon)$, where  $C \in  \C$, and $\Delta:C \to C \otimes C$ (comultiplication) and  $\varepsilon:C \to \unit$ (counit) are morphisms in $\C$,
    satisfying counitality and coassociativity constraints: 
    \begin{align*}
      \hspace{.5in}  \alpha_{C,C,C}(\Delta \otimes \id) \Delta = (\id \otimes \Delta)\Delta, \quad \quad 
        (\varepsilon \otimes \id)\Delta = l_C^{-1}, \quad  \quad (\id \otimes \varepsilon)\Delta = r_C^{-1}.
    \end{align*}
  A {\it morphism} of coalgebras $(C, \Delta_C, \varepsilon_C)$ to $(D, \Delta_{D}, \varepsilon_{D})$  is a morphism \linebreak $g: C \to D$ in $\C$ so that $\Delta_{D}g  = (g \otimes g) \Delta_{C}$ and $\varepsilon_{D} g = \varepsilon_C$. Coalgebras in $\C$ and their morphisms  form a category, which we denote by ${\sf Coalg}(\C)$.
  
   \medskip
   
    \item A \textit{Frobenius algebra} in $\C$ is a tuple $(A, m, u, \Delta, \varepsilon)$, where
   $(A,m,u) \in {\sf Alg}(\C)$ and $(A,\Delta,\varepsilon) \in {\sf Coalg}(\C)$, so that 
   
    \vspace{-.15in}
    
       \begin{equation*}
      \hspace{.5in} (m \otimes \id)\alpha^{-1}_{A,A,A}(\id \otimes \Delta)=\Delta m = (\id \otimes m)\alpha_{A,A,A}(\Delta \otimes \id).
      \end{equation*}
      
      \smallskip
      
   \noindent  A {\it morphism} of Frobenius algebras in $\C$ is a morphism in $\C$ that lies in both ${\sf Alg}(\C)$ and ${\sf Coalg}(\C)$. 
    Frobenius algebras in $\C$ and their morphisms  form a category, which we denote by ${\sf FrobAlg}(\C)$.
\end{enumerate}
\end{definition}

\begin{remark} \label{rem:Frobpairing}
\begin{enumerate}
    \item Alternatively, a Frobenius algebra in $\mathcal{C}$ is a tuple $(A,m,u,p,q)$, where $(A, m, u) \in {\sf Alg}(\C)$, $p: A \otimes A \to \unit$  and $q: \unit \to A \otimes A$  are morphisms in $\mathcal{C}$ satisfying an invariance condition, $p(\id_A \otimes m)\alpha_{A,A,A} = p(m \otimes \id_A)$, and the `snake' equations.
To convert from $(A,m,u,p,q)$ to $(A,m,u,\Delta,\varepsilon)$ in Definition~\ref{def:algstr}(c), take 
$\Delta:=(m \otimes \id_A)\alpha_{A,A,A}^{-1}(\id_A \otimes q)r_A^{-1}$ and  $\varepsilon:=p\;(u \otimes \id_A)r_A^{-1}.$
On the other hand, to convert from $(A,m,u,\Delta,\varepsilon)$ to $(A,m,u,p,q)$, take 
$p:= \varepsilon_A m_A$ and $q:= \Delta_A u_A$. 

\smallskip

\item Note that ${}^* \hspace{-.03in}A$ is naturally a left $A$-module. A Frobenius algebra in $\C$ can then be equivalently defined as an algebra $A$ in $\C$ so that $(A,\lambda_A)$ is isomorphic to $({}^* \hspace{-.03in}A, \lambda_{{}^* \hspace{-.03in}A})$ as left $A$-modules. See Definition \ref{def:A-mod} below for the definition of an $A$-module. 
\end{enumerate}

\smallskip
See \cite{FuchsStigner} an \cite[Section~2.3]{kock2003frobenius} for more details.
\end{remark}

Next, we recall how the functors of Definition~\ref{def:monfunc} preserve the algebraic structures in  Definition~\ref{def:algstr}.

\begin{proposition} \label{prp:monfunctor}
\cite[p.100-101]{Str} \cite[Lemma ~2.1]{Sz} \cite[Corollary~5]{DP08} \cite[Prop.~2.13]{KongRunkel}
Let $\C$ and $\mathcal D$ be monoidal categories.
\begin{enumerate}
    \item[\textnormal{(a)}]   Let $(F, F_{\ast,\ast}, F_0): \C \to \mathcal{D}$ be a monoidal functor. If $(A,m,u) \in {\sf Alg}(\C)$, then $$(F(A),~ F(m)F_{A,A}, ~F(u)F_0) \in {\sf Alg}(\mathcal{D}).$$
    \item[\textnormal{(b)}]   Let $(F, F^{\ast,\ast},F^0): \C \to \mathcal{D}$ be a comonoidal functor. If $(C,\Delta,\varepsilon) \in {\sf Coalg}(\C)$, then 
    $$(F(C),~ F^{C,C}F(\Delta), ~F^0F(\varepsilon)) \in {\sf Coalg}(\mathcal{D}).$$ 
    \item[\textnormal{(c)}] Let $(F, F_{\ast,\ast},F_0,F^{\ast,\ast},F^0):\C \to \mathcal{D}$ be a  Frobenius monoidal functor. If $(A,m,u,\Delta,\varepsilon) \in {\sf FrobAlg}(\C)$, then 
    $$(F(A),\; F(m)F_{A,A}, \; F(u)F_0, \; F^{A,A}F(\Delta),\;  F^0F(\varepsilon)) \in {\sf FrobAlg}(\mathcal{D}).$$  
    
    \vspace{-.3in}
    \qed
\end{enumerate}
\end{proposition}

Some properties of the structures in Definition~\ref{def:algstr} of interest are given below.

\begin{definition}
\label{def:algprop}  Take $\C$ a  fusion category.
\begin{enumerate}
     \item $A \in {\sf Alg}(\C)$ is \textit{indecomposable} if it is not isomorphic to a direct sum of non-trivial algebras in $\C$.
    
     \smallskip
     
     \item $A \in {\sf Alg}(\C)$ is \textit{connected} (or \textit{haploid}) if 
     $\dim_{\kk}\Hom_{\mathcal C} (\unit, A) = 1$.

\smallskip
    
        \item $A \in {\sf Alg}(\C)$ is \textit{separable} if there exists a morphism $\Delta':A \to A \otimes A$ in $\C$ so that  $m \Delta' = \id_A$ as maps in $\C$ with
              \vspace{-.15in}
              
    \begin{equation*}
          \hspace{.5in}
               (\id_A\otimes m) \alpha_{A, A, A} (\Delta'\otimes \id_A) ~=~ \Delta' m ~=~ (m\otimes \id_A) \alpha^{-1}_{A, A, A} (\id_A\otimes \Delta').
   \end{equation*}
   
   \smallskip
   
       \item $(A,m,u,\Delta,\varepsilon) \in {\sf FrobAlg}(\C)$ is 
       \textit{special} if $m \Delta =  \id_A$ and $\varepsilon u = \varphi \; \id_\unit$ for a nonzero $\varphi \in \kk$.
     \end{enumerate}
\end{definition}

\begin{remark} \label{rmk:properties}
\begin{enumerate}
  \item The displayed equations in  Definition~\ref{def:algprop}(c) above are that $m$ splits as a map of $A$-bimodules in $\C$ via $\Delta'$; see Sections~\ref{sec:mod} and~\ref{sec:bimod}.
  
  \smallskip
  
 \item  The special Frobenius condition above implies separability, and connected implies indecomposable.
 \end{enumerate}
\end{remark}


\subsection{Categories of modules over algebras} \label{sec:mod} Fix $\C$ a fusion category. Now we turn our attention to modules over algebras in $\C$. For more details, see \cite[Section~3]{OstrikTG2003} and \cite[Section~7.8]{EGNO}.

\begin{definition}[$\rho_M$, $\rho_M^A$, $\lambda_M$, $\lambda_M^A$, $\C_A$, ${}_A \C$] \label{def:A-mod}  Take $A:=(A,m_A,u_A)$, an algebra in $\C$. 
A {\it right $A$-module in $\C$} is a pair $(M, \rho_M)$, where $M \in \C$, and $\rho_M:=\linebreak \rho_M^A: M \otimes A \to M$ is a morphism in $\C$  so that $$\rho_M(\rho_M \otimes \id_A) = \rho_M(\id_M \otimes m_A)\alpha_{M,A,A} \quad \text{ and } \quad r_M = \rho_M(\id_M \otimes u_A).$$ A {\it morphism} of right $A$-modules in $\C$ is a morphism $f: M \to N$ in $\C$ so that $f\rho_M  = \rho_N (f \otimes \id_A)$. Right $A$-modules in $\C$ and their morphisms form a category, which we denote by~$\C_A$.
The category ${}_A \C$ of {\it left $A$-modules $(M, \lambda_M:=\lambda_M^A:A \otimes M \to M)$  in $\C$} is defined likewise.
\end{definition}

We have that $\C_A$ is a left $\C$-module category: for $X \in \C$ and $(M, \rho_M) \in \C_A$, the bifunctor $\C \times \C_A \to \C_A$ is defined by 
$$(X \otimes M) \otimes A \overset{\alpha_{X,M,A}}{\xrightarrow{\hspace*{2cm}}} X \otimes (M \otimes A) \overset{\id_X \otimes \rho_M}{\xrightarrow{\hspace*{2cm}}} X \otimes M.$$
Similarly,  ${}_A \C$ is a right $\C$-module category.

\begin{proposition}\label{prop:indssCA} \cite[Remark~3.1]{OstrikTG2003} \cite[Proposition~7.8.30]{EGNO}
We have that $\mathcal{C}_A$ is an indecomposable (resp., semisimple) $\mathcal{C}$-module category if $A$ is an indecomposable (resp., separable) algebra in $\mathcal{C}$. \qed
\end{proposition}

Next, we turn our attention to Morita equivalence of algebras in fusion categories.

\begin{definition}  \label{def:Moritaequiv} We say that two algebras $A$ and $B$ in $\C$ are {\it Morita equivalent} if $\C_A \sim \C_B$ as (left) $\C$-module categories.
\end{definition}

Several algebraic properties are preserved under Morita equivalence, such as indecomposability and separability. We will discuss a characterization of Morita equivalence in terms of bimodules in Section~\ref{sec:Moritabimod}.


\subsection{Categories of bimodules over algebras} \label{sec:bimod}
We recall here preliminary notions on bimodules over algebras in a fusion category $\C$. For general information, see \cite[Section~3.3]{Mombelli} and \cite[Section~7.8]{EGNO}.

\begin{definition}[${}_A \C_A$] \label{def:bimod} Take $A:=(A,m_A,u_A) \in {\sf Alg}(\C)$. 
An {\it $A$-bimodule in $\C$} is a triple $(M, \lambda_M, \rho_M)$, where $M \in \C$, and $\lambda_M: A \otimes M \to M$ and $\rho_M: M \otimes A \to M$ are morphisms in $\C$, so that $(M, \lambda_M) \in {}_A \C$ and $(M, \rho_M) \in \C_A$ with $$\lambda_M(\id_A \otimes \rho_M)\alpha_{A,M,A} ~=~ \rho_M(\lambda_M \otimes \id_A).$$ A {\it morphism} of  $A$-bimodules in $\C$ is a morphism $f: M \to N$ in $\C$ that is simultaneously a morphism in both ${}_A \C$ and $\C_A$. Bimodules over $A$ in $\C$ and their morphisms form a category, which we denote by ${}_A \C_A$. 
\end{definition}

\begin{definition}[$\otimes_A$, $\pi_{M,N}, \pi^A_{M,N}$] \label{def:pi}
Take $A$-bimodules $M$ and $N$ in $\C$. The {\it tensor product} of $M$ and $N$ over $A$ is the object of ${}_A \C_A$ given by
$$M \otimes_A N := \textnormal{coker}\big(\rho_M \otimes \id_N - (\id_M \otimes \lambda_N)\alpha_{M,A,N}\big).$$
Let $\pi_{M,N}:=\pi^A_{M,N}: M \otimes N \to M \otimes_A N$ denote the canonical projection, a morphism in $\C$.
Moreover, $M \otimes_A N$ is an $A$-bimodule via morphisms:
$$\lambda_{M\otimes_A N}: A \otimes (M\otimes_A N)\to M\otimes_A N \; \; \text{and}\; \; \rho_{M\otimes_A N}: (M\otimes_A N)\otimes A\to M\otimes_A N$$
so that 
$$\lambda_{M\otimes_A N}(\id_A\otimes\pi_{M, N})=\pi_{M, N}(\lambda_M\otimes \id_N)\alpha^{-1}_{A, M, N},$$
$$\pi_{M, N}(\id_M\otimes \rho_N)\alpha_{M, N, A}=\rho_{M\otimes_A N}(\pi_{M, N}\otimes \id_A).$$
\end{definition}

\smallskip

\begin{proposition}[$({}_A \mathcal{C}_A, ~\otimes_A,  ~A, ~\alpha_{\ast,\ast,\ast}^A, ~l_\ast^A, ~r_\ast^A)$] \label{prop:ACAmonoidal} \cite[Section~3.3.2]{Mombelli}
The category ${}_A \mathcal{C}_A$ has the structure of a monoidal category with
\begin{itemize}
    \item tensor product $\otimes_A$, 
    \item unit object $A$, and
    \item associativity constraint $\alpha^A_{X,X',X''}: (X \otimes_A X') \otimes_A X'' \overset{\sim}{\to} X \otimes_A (X' \otimes_A X'')$ for $X,X',X'' \in \C$, so that 
    $$ \hspace{.3in} \alpha^A_{X,X',X''}\; \pi_{X \otimes_A X', X''} \; (\pi_{X,X'} \otimes \id_{X''}) = \pi_{X, X' \otimes_A X''}\; (\id_X \otimes \pi_{X',X''})\; \alpha_{X,X',X''},$$
    \item unit constraints $l^A_X: A \otimes_A X \overset{\sim}{\to} X$ and $r^A_X: X \otimes_A A \overset{\sim}{\to} X$ so that 
    $$l^A_X \; \pi_{A,X} = \lambda_X \qquad \text{and} \qquad r^A_X \; \pi_{X,A} = \rho_X.$$ 
\end{itemize}

\vspace{-.2in}
    \qed
\end{proposition}

 In addition, for maps $f:X \to W$ and $g: Y \to Z$ in ${}_A \mathcal{C}_A$, we get that
\begin{equation} \label{eq:tenAmap}
(f \otimes_A g) \; \pi_{X,Y} ~=~ \pi_{W,Z} \; (f \otimes g)
\end{equation}
as maps in $\mathcal{C}$.

Moreover, we have by a result of Yamagami that ${}_A \mathcal{C}_A$ is a fusion category under nice conditions on $A$. 

\begin{proposition} \label{prop:Yam} 
\cite[Proposition~5.6, Corollary~6.2]{Yamagami} If $A$ is an indecomposable, special Frobenius algebra in a fusion category $\C$, then ${}_A \mathcal{C}_A$  is a fusion category. \qed
\end{proposition}


\subsection{On Morita equivalence of algebras} \label{sec:Moritabimod}

We provide here characterizations for the Morita equivalence of algebras in fusion categories [Definition~\ref{def:Moritaequiv}], and provide other preliminary results that we will need later in Section~\ref{sec:MoritaPhi}. First, consider the following notation. 

\begin{definition}[$\overline{\alpha}_{\ast,\ast,\ast}$]  \label{def:alphabar}
Let $\mathcal{C}$ be a fusion category, and take two algebras $A$ and $B$ in $\C$. Let $X,Z \in {}_A \C_B$ and $Y \in {}_B \C_A$. Take 
\[
\begin{array}{c}
\medskip
\overline{\alpha}_{X,Y,Z}: (X \otimes_B Y) \otimes_A Z \to X \otimes_B (Y \otimes_A Z)
\end{array}
\]
to be the morphism  in $\C$ defined by the commutative diagram:
\[
{\small
\xymatrix@C+3pc@R-.5pc{
(X \otimes Y) \otimes Z \ar[r]^{\pi^B_{X,Y} \otimes \id_Z} \ar[d]_\alpha& (X \otimes_B Y) \otimes Z  \ar[r]^{\pi^A_{XY,Z}} 
& (X \otimes_B Y) \otimes_A Z \ar@{-->}[d]^{\overline{\alpha}}\\
X \otimes (Y \otimes Z)\ar[r]^{\pi^B_{X,YZ}} & X \otimes_B (Y \otimes Z) \ar[r]^{\id_X \otimes_B \pi^A_{Y,Z}} & X \otimes_B (Y \otimes_A Z).\\
}
}
\]
The same notation will apply in the case when the roles $A$ and $B$ are reversed.
\end{definition}

\begin{lemma} \label{lem:alphabar} \cite[Exercise~7.8.28]{EGNO}
 The morphism $\overline{\alpha}$ exists, and is an isomorphism in $\C$.  \qed
\end{lemma}

\begin{proposition}
\label{prop:Morita-bimod}
Take two algebras $A$ and $B$ in a fusion category $\C$. Then the following statements hold.
\begin{enumerate}
    \item[(a)] $A$ and $B$ are Morita equivalent  if and only if there exist bimodules $P \in {}_A \C_B$ and $Q \in {}_B \C_A$ so that $P \otimes_B Q \cong A$ in ${}_A \C_A$ and $Q \otimes_A P \cong  B$ in ${}_B \C_B$.
    \smallskip
    \item[(b)] If there exist bimodules $P \in {}_A \C_B$ and $Q \in {}_B \C_A$ along with epimorphisms 
    $$\tau: P \otimes_B Q \twoheadrightarrow A \; \text{in} \; {}_A \C_A \quad \quad \text{and} \quad \quad \mu: Q \otimes_A P \twoheadrightarrow  B \; \text{in} \; {}_B \C_B$$ so that the diagrams $(\ast)$ and $(\ast \ast)$ below commute in $\C$, then the equivalent conditions of part (a) hold.
\end{enumerate}
{\footnotesize
\[
\xymatrix@R-.5pc@C-1.1pc{
(P \otimes_B Q) \otimes_A P \ar[rr]^{\overline{\alpha}} \ar@{->>}[d]^{\tau \otimes_A \id_P} && P \otimes_B (Q \otimes_A P) \ar@{->>}[d]_{\id_P \otimes_B \mu}\\
A \otimes_A P \ar[dr]_{l_P^A} & (\ast) & P \otimes_B B \ar[dl]^{r_P^B}\\
& P&
}
\hspace{.25in}
\xymatrix@R-.5pc@C-1.1pc{
(Q \otimes_A P) \otimes_B Q \ar[rr]^{\overline{\alpha}} \ar@{->>}[d]^{\mu \otimes_B \id_Q} && Q \otimes_A (P \otimes_B Q) \ar@{->>}[d]_{\id_Q \otimes_A \tau}\\
B \otimes_B Q \ar[dr]_{l_Q^B} & (\ast \ast)& Q \otimes_A A \ar[dl]^{r_Q^A}\\
& Q&
}
\]
}
\end{proposition}

\begin{proof}
(a) This is well-known; see, e.g., \cite[Remark~3.2]{OstrikTG2003} and \cite{FRS-2}.

\smallskip

(b)  Since $\C$ is assumed to be fusion, the category ${}_A \C_A$ is also abelian (see, e.g., \cite[Exercise~7.8.7]{EGNO}). So it suffices to show $\tau$ and $\mu$ are monomorphisms in ${}_A \C_A$ as epic monomorphisms are isomorphisms in abelian categories.  We prove the statement for $\tau$; the proof for $\mu$ will follow similarly.

Take morphisms $g_1, g_2: W \to P \otimes_B Q$ in ${}_A \C_A$ so that $\tau g_1 = \tau g_2$ as morphisms $W \to A$ in ${}_A \C_A$. Consider the following commutative diagram in $\C$, where we suppress the  $\otimes$ symbol in morphisms. We also invoke Lemma~\ref{lem:alphabar} in all of the diagrams below for the existence of the morphism  $\overline{\alpha}$.

{\scriptsize
\[
\hspace{-.1in}
\xymatrix@C-4.5pc@R-.8pc{
(W \otimes_A P) \otimes_B Q \ar[rrrr]^{\overline{\alpha}} \ar[ddd]_{g_i \; \id \; \id} & & & & W \otimes_A (P \otimes_B Q) \ar[r]^{\hspace{.18in}\id \tau} \ar[d]_{g_i \; \id \; \id} & W \otimes_A A \ar[ddd]^{g_i \; \id}\\
& & & &(P \otimes_B Q) \otimes_A (P \otimes_B Q)\ar[d]^{\overline{\alpha}}\ar[rdd]^{\id \; \id \; \tau} & \\
&(P \otimes_B (Q \otimes_A P)) \otimes_B Q \ar[rrd]^{\overline{\alpha}} \ar[dd]^(.7){\id \; \mu \; \id} &&&P \otimes_B (Q \otimes_A (P \otimes_B Q)) \hspace{.5in} \ar[ld]_(.6){\id \; \overline{\alpha}^{-1}} &\\
((P \otimes_B Q) \otimes_A P) \otimes_B Q  \ar[ur]^{\overline{\alpha} \; \id} \ar[dd]_{\tau \; \id \; \id}&  && P \otimes_B ((Q \otimes_A P) \otimes_B Q) \ar[dr]^{\id \; \overline{\alpha}} \ar[dd]_{\id \; \mu \; \id} && (P \otimes_B Q) \otimes_A A \ar[dd]^{\overline{\alpha}}\\
\hspace{.7in} (\ast)&(P \otimes_B B)\otimes_B Q \ar[ddr]_(.4){r_P^B \; \id} \ar[drr]^{\overline{\alpha}} &&  & P \otimes_B (Q \otimes_A (P \otimes_B Q)) \ar[dr]_{ \id \; \id \; \tau}&\\
(A \otimes_A P) \otimes_B Q \ar[drr]_{l_P^A \; \id} &&&P \otimes_B (B \otimes_B Q) \ar[dl]_(.65){\id \; l_Q^B} &(\ast \ast) \hspace{.3in} & P \otimes_B (Q \otimes_A A) \ar[dlll]^{\id \; r_Q^A}\\
&& P \otimes_B Q &&&}
\]
}

Now with the diagram commuting, the assumption $\tau g_1 = \tau g_2$ implies that 
\[
\begin{array}{l}
\smallskip
(\id_P \otimes_B r_Q^A)\;\overline{\alpha}_{P,Q,A}\;(g_1 \otimes_A \id_A)(\id_W \otimes_A \tau)\;\overline{\alpha}_{W,P,Q}\\
= (\id_P \otimes_B r_Q^A)\;\overline{\alpha}_{P,Q,A}\;(g_2 \otimes_A \id_A)(\id_W \otimes_A \tau)\;\overline{\alpha}_{W,P,Q}.
\end{array}
\]
Note that $\id_W \otimes_A \tau$ is an epimorphism as $\tau$ is epic and $\otimes_A$ is right exact in each variable \cite[Exercise 7.8.23]{EGNO}. Therefore, since $\overline{\alpha}$ is an epimorphism by Lemma~\ref{lem:alphabar}, we get that
$$(\id_P \otimes_B r_Q^A)\;\overline{\alpha}_{P,Q,A}\;(g_1 \otimes_A \id_A)
= (\id_P \otimes_B r_Q^A)\;\overline{\alpha}_{P,Q,A}\;(g_2 \otimes_A \id_A).$$
By \cite[Exercise~7.8.22]{EGNO} we have that $r_Q^A$ is an isomorphism in ${}_B \mathcal{C}$, so $\id_P \otimes_B r_Q^A$ is an isomorphism in $\mathcal{C}$ as well. Therefore,
$$\overline{\alpha}_{P,Q,A}\;(g_1 \otimes_A \id_A)
= \overline{\alpha}_{P,Q,A}\;(g_2 \otimes_A \id_A).$$
Finally, by Lemma~\ref{lem:alphabar}, $\overline{\alpha}$ is an isomorphism.
Thus, $(g_1 \otimes_A \id_A) = (g_2 \otimes_A \id_A)$, and $g_1 = g_2$, as desired.
\end{proof}

Part (a) is a generalization of a classical ring theory result, which is presented, e.g., in \cite[Theorem~4.4.5]{Cohn}. The proof of (b) is a generalization of  \cite[Lemma~4.5.2]{Cohn}. Moreover, the result below generalizes the classic result that a $\kk$-algebra $R$ is Morita equivalent to a matrix algebra Mat$_n(R)$ over $\kk$. (Indeed, the classical result is recovered from the following result by letting $\C$ be the fusion category of finite-dimensional $\kk$-vector spaces with $S = R$ and $V  = \kk^{\oplus n}$.)

\begin{proposition} \label{prop:MoritaMat}
Let $\C$ be a fusion category, and take an algebra $S$ in $\C$  and an object $V$ in $\C$. Then,
\begin{enumerate}
   \item  $({}^* \hspace{-.02in} {V}\otimes S) \otimes {V} \in {\sf Alg}(\C)$ with
    $$m_{({}^*\hspace{-.02in}{V} \otimes S) \otimes {V}} = (\id_{{}^*\hspace{-.02in}{V}} \otimes m_S \otimes \id_{V})\alpha_2(r_{{}^*\hspace{-.02in}{V} \otimes S} \otimes \id_{S {V}})(\id_{{}^*\hspace{-.02in}{V} S} \otimes {\sf ev}'_{V} \otimes \id_{S {V}})\alpha_1, \text{ for} $$ 
    \[
    \begin{array}{rl}
    \alpha_1 &=\alpha^{-1}_{{}^*\hspace{-.02in}{V}S,{V} {}^*\hspace{-.02in}{V}, S{V}} (\id_{{}^*\hspace{-.02in}{V} S} \otimes \alpha^{-1}_{{V},{}^*\hspace{-.02in}{V},S{V}})(\id_{{}^*\hspace{-.02in}{V} S {V}} \otimes \alpha_{{}^*\hspace{-.02in}{V},S,{V}})\alpha_{{}^*\hspace{-.02in}{V} S,{V}, {}^*\hspace{-.02in}{V} S{V}},\\
    \alpha_2 &= (\alpha_{{}^*\hspace{-.02in}{V},S,S} \otimes \id_{V})\alpha^{-1}_{{}^*\hspace{-.02in}{V} S, S, {V}}, \text{ and}
    \end{array}
    \]
    $$u_{({}^*\hspace{-.02in}{V} \otimes S) \otimes {V}} =
    (\id_{{}^*\hspace{-.02in}{V}} \otimes u_S \otimes \id_{V})(r_{{}^*\hspace{-.02in}{V}}^{-1} \otimes \id_{V}){\sf coev}'_{V};$$
    \item $S$ and $({}^* \hspace{-.02in}{V} \otimes S) \otimes {V}$ are Morita equivalent as algebras in $\C$.
\end{enumerate}
\end{proposition}

\begin{proof}
(a) We leave this to the reader. 

\smallskip

(b) Let $T$ denote the algebra $({}^*\hspace{-.02in}{V} \otimes S) \otimes {V}$ in part (a). Let $P:=  {}^*\hspace{-.02in}{V} \otimes S$ and $Q:= S \otimes {V}$. It follows from the associativity of $m_S$, and naturality of  $\alpha$ and  $r$, that the morphisms

\[\begin{array}{rl}
\smallskip
\lambda_P^T &= (\id_{{}^*\hspace{-.02in}{V}} \otimes m_S)(\id_{{}^*\hspace{-.02in}{V} S} \otimes r_S \otimes \id_S)(\id_{{}^*\hspace{-.02in}{V} S} \otimes {\sf ev}'_{V} \otimes \id_S) \alpha_3\\
\smallskip
\rho_P^S &= (\id_{{}^*\hspace{-.02in}{V}} \otimes m_S)\alpha_{{}^*\hspace{-.02in}{V},S,S}\\
\smallskip
\lambda_Q^S &= (m_S \otimes \id_{V})\alpha^{-1}_{S,S,{V}}\\
\rho_Q^T &=  (m_S \otimes \id_{V})(r_S \otimes \id_{S{V}})(\id_S \otimes {\sf ev}'_{V} \otimes \id_{S{V}}) \alpha_4,\\
\end{array}\]

for 
\[
\begin{array}{rl}
\smallskip
\alpha_3 &= (\id_{{}^*\hspace{-.02in}{V}} \otimes \alpha_{S,{V},{}^*\hspace{-.02in}{V}} \otimes \id_S)(\id_{{}^*\hspace{-.02in}{V}} \otimes \alpha^{-1}_{S{V}, {}^*\hspace{-.02in}{V}, S})\alpha_{{}^*\hspace{-.02in}{V},S{V},{}^*\hspace{-.02in}{V} S} (\alpha_{{}^*\hspace{-.02in}{V}, S,{V}} \otimes \id_{{}^*\hspace{-.02in}{V} S})\\
\alpha_4 &= (\alpha_{S,{V},{}^*\hspace{-.02in}{V}} \otimes \id_{S {V}})(\alpha^{-1}_{S{V}, {}^*\hspace{-.02in}{V}, S} \otimes \id_{V})\alpha^{-1}_{S{V},{}^*\hspace{-.02in}{V} S,{V}},
\end{array}
\] 
imply that $(P, \lambda_P^T,  \rho_P^S) \in {}_T \C_S$ and  $(Q, \lambda_Q^S,  \rho_Q^T) \in {}_S \C_T$. Moreover, consider the morphisms
\smallskip
\[
\begin{array}{rrl}
\smallskip
\widehat{\tau} = &(\id_{{}^*\hspace{-.02in}{V}} \otimes m_S \otimes \id_{V}) (\alpha_{{}^*\hspace{-.02in}{V}, S, S} \otimes \id_{V}) \alpha^{-1}_{{}^*\hspace{-.02in}{V} S, S, {V}} &: P \otimes Q \to T,\\
\widehat{\mu} = &m_S (r_S \otimes \id_S)(\id_S \otimes {\sf ev}'_{V} \otimes \id_S) (\alpha_{S, {V}, {}^*\hspace{-.02in}{V}} \otimes \id_S) \alpha^{-1}_{S{V}, {}^*\hspace{-.02in}{V}, S} &: Q \otimes P \to S.
\end{array}
\]

\smallskip

\noindent It follows from the associativity of $m_S$, and naturality of $\alpha$ and $r$, that $\widehat{\tau} \in {}_T \C_T$ and $\widehat{\mu} \in {}_S \C_S$. It is also clear that $\widehat{\tau}$ and $\widehat{\mu}$ are epimorphisms in $\C$. Moreover, the morphisms factor through epimorphisms $\tau: P \otimes_S Q \to T$ and $\mu: Q \otimes_T P \to S$, respectively, so that $\widehat{\tau} = \tau\; \pi_{P,Q}^S$ and  $\widehat{\mu} = \mu\; \pi_{Q,P}^T$. Indeed, by the naturality of  $\alpha$ and the associativity of $m_S$, we get that $\widehat{\tau}(\rho_P^S \otimes \id_Q) = \widehat{\tau}(\id_P \otimes \lambda_Q^S)\alpha_{P,S,Q}$. So the claim for $\tau$ follows from the definition of $P \otimes_S Q$. Likewise, the claim for $\mu$ holds. 

\pagebreak 

Finally, by Proposition~\ref{prop:Morita-bimod}(b), it suffices to show that $\tau$ and $\mu$ satisfy the diagrams ($\ast$) and ($\ast \ast$) there. We will do so for ($\ast$) in the strict case, and the general case, along with ($\ast \ast$) will hold in a similar manner. The unadorned $\otimes$ symbol in morphisms are suppressed below.

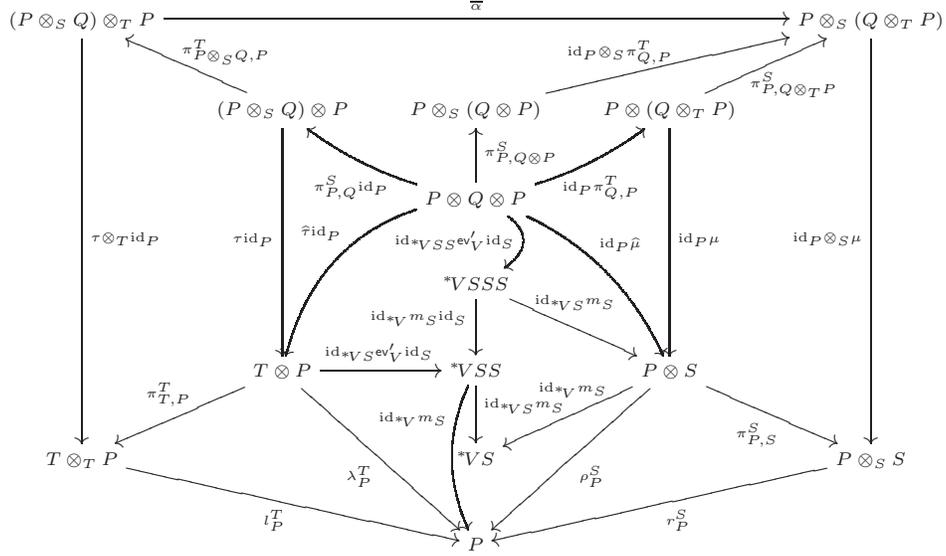
\[
{\scriptsize
\xymatrix@R-.3pc@C-.6pc{
\left( P \otimes_S Q \right) \otimes_T P \ar[rrrr]^{\overline{\alpha}} \ar[ddddd]^{\tau \otimes_T \id_P} &&&& P \otimes_S \left( Q \otimes_T P \right) \ar[ddddd]_{\id_P \otimes_S \mu}\\
& \left( P \otimes_S Q \right) \otimes P \ar[ul]_(.4){\pi^T_{P \otimes_S Q,P}} \ar[ddd]_{\tau \id_P} &  P \otimes_S \left( Q \otimes P \right) \ar[urr]^(.4){\id_P \otimes_S \pi_{Q,P}^T} & P \otimes \left( Q \otimes_T P \right) \ar[ur]_{\pi_{P,Q \otimes_T P}^S} \ar[ddd]^{\id_P \mu}& \\
&& P \otimes Q \otimes P \ar@/^1pc/[ddr]^{\id_P \widehat{\mu}} \ar@/^.5pc/[ul]^{\pi_{P,Q}^S \id_P} \ar[u]_{\pi_{P,Q \otimes P}^S} \ar@/_.5pc/[ur]_{\id_P \pi_{Q,P}^T} \ar@/_1.5pc/[ddl]_{\widehat{\tau} \id_P} \ar@/^1.5pc/[d]_{\id_{{}^* \hspace{-.02in}VSS} {\sf ev}'_V \id_S}  && \\
&& {}^*\hspace{-.02in}VSSS \ar[d]_(.4){\id_{{}^*\hspace{-.02in}V} m_S \id_S} \ar[dr]^(.4){\id_{{}^*\hspace{-.02in}VS} m_S} &&\\
& T \otimes P \ar[r]^{\id_{{}^*\hspace{-.02in}VS} {\sf ev}'_V \id_S} \ar[ddr]_{\lambda_P^T} \ar[dl]_{\pi_{T,P}^T} & {}^*\hspace{-.02in}VSS \ar[d]^(.4){\id_{{}^*\hspace{-.02in}VS }m_S} \ar@/_.75pc/[dd]_(.3){\id_{{}^*\hspace{-.02in}V} m_S}& P \otimes S \ar[dl]_(.4){\id_{{}^*\hspace{-.02in}V} m_S} \ar[ddl]^{\rho_P^S} \ar[dr]_{\pi_{P,S}^S} & \\
T \otimes_T P \ar[drr]_{l_P^T} && {}^*\hspace{-.02in}VS && P \otimes_S S \ar[dll]^{r_P^S} \\
&& P && \\
}}
\]

All regions commute either by the definitions of the maps involved, 
by~\eqref{eq:tenAmap}, or by the associativity of $m_S$.
\end{proof}


\section{A Frobenius monoidal functor $\Phi$ to a category of bimodules} \label{sec:Phi}

 Our main result in this section is that, when $A$ is a special Frobenius algebra in a fusion category $\C$, we  endow the free functor from  $\mathcal{C}$ to the category of $A$-bimodules in $\mathcal{C}$ with a Frobenius monoidal structure. Consider the notation below.

\begin{notation}[$\tilde{\ast}$] \label{not:tilde} 
 Take $A \in {\sf Alg}(\C)$, and take objects $X, X', W \in {}_A \mathcal{C}_A$. For a map $f:X \otimes_A X' \to W$ in ${}_A \mathcal{C}_A$, let $\tilde{f}: X \otimes X' \to W$ denote its lift in $\mathcal{C}$ in the sense that $\tilde{f} = f \; \pi_{X,X'}$.
\end{notation}

Now, we have the following result.

\begin{theorem}[$\Phi$] \label{thm:Phi} 
Take $\C$ to be a fusion category and let $A =(A, m_A,u_A, \Delta_A, \varepsilon_A)$  be a special Frobenius algebra in $\C$.  Then the following functor is Frobenius monoidal:
\begin{align*}
\Phi: \mathcal{C} &\to {}_A \mathcal{C}_A \\
X &\mapsto (A \otimes X) \otimes A \quad \quad \; \text{(as objects)}\\
\varphi &\mapsto (\id_A \otimes \varphi)  \otimes \id_A \quad \text{(as morphisms)}.
\end{align*}
Here, the monoidal structure $\Phi_{X,X'}$ is defined by the lift of $\widetilde{\Phi}_{X,X'}$, that is, $$\widetilde{\Phi}_{X,X'}=\Phi_{X,X'} \; \pi_{\Phi(X),\Phi(X')},$$ with:
\begin{align*}
\widetilde{\Phi}_{X,X'} &= \;
(\alpha_{A,X,X'} \otimes \id_A) (\id_{AX} \otimes \; l_{X'} \otimes \id_A)   (\id_{AX} \otimes   \varepsilon_A  m_A \otimes \id_{X',A})   
 \underline{\alpha},
 \quad \text{for} \\
\underline{\alpha} &:= \; (\id_{AX} \otimes \alpha^{-1}_{A,A,X'} \otimes \id_A) (\alpha_{AX,A,AX'} \otimes \id_A) \alpha^{-1}_{AXA,AX',A},
\end{align*}
and by
$
\Phi_0 = 
(r_A^{-1} \otimes \id_A)  \Delta_A.
$

\smallskip

\noindent Moreover, the comonoidal structure $\Phi^{X,X'}=\pi_{\Phi(X), \Phi(X^\prime)} \; \widetilde{\Phi}^{X,X'}$ is given by 
\begin{align*}
\widetilde{\Phi}^{X,X'} &= \;
\underline{\alpha'} (\id_{AX} \otimes \Delta_Au_A \otimes \id_{X',A})  (\id_{A,X} \otimes l_{X'}^{-1} \otimes \id_A)  (\alpha^{-1}_{A,X,X'} \otimes \id_A),
 \quad \text{for} \\
\underline{\alpha'} &:= \alpha_{AXA,AX',A}  (\alpha^{-1}_{AX,A,AX'} \otimes \id_A)  (\id_{AX} \otimes \alpha_{A,A,X'} \otimes \id_A),
\end{align*}
and by
$
\Phi^0 = m_A  (r_A \otimes \id_A).
$
\end{theorem}

\begin{proof} 
We need to verify the following conditions:
\begin{itemize}
    \item[(a)] $\Phi(X)$ is an $A$-bimodule in $\mathcal{C}$;
    \item[(b)] $\Phi_{X, X^\prime}$ is well defined via $\widetilde{\Phi}_{X, X^\prime}$, that is,
    \begin{itemize}
        \item[(b.1)] $\widetilde{\Phi}_{X, X^\prime}\; (\rho^A_{\Phi(X)}\otimes \id_{\Phi(X^\prime)})=\widetilde{\Phi}_{X, X^\prime}\; (\id_{\Phi(X)}\otimes \lambda^A_{\Phi(X^\prime)}) \; \alpha_{\Phi(X), A, \Phi(X^\prime)}$, and
        \item[(b.2)]  $\Phi_{X, X^\prime}$ is an $A$-bimodule map;
    \end{itemize} 
    \item[(c)] $\Phi_0$, $\Phi^{X, X^\prime}$, $\Phi^0$ are $A$-bimodule maps;
    \item[(d)] the associativity, unitality,  coassociativity, and counitality axioms;
    \item[(e)] the Frobenius conditions:
\begin{align*}
&(\Phi_{X,X^\prime} \otimes_A \id_{\Phi(X^{\prime\prime})}) \; (\alpha^A_{\Phi(X),\Phi(X^\prime),\Phi(X^{\prime\prime})})^{-1} \; (\id_{\Phi(X)} \otimes_A \Phi^{X^\prime,X^{\prime\prime}})\\ &=\Phi^{XX^\prime,X^{\prime\prime}} \; \Phi(\alpha_{X,X\prime,X^{\prime\prime}}^{-1}) \; \Phi_{X,X^\prime X^{\prime\prime}}, \\
&(\id_{\Phi(X)} \otimes_A \Phi_{X^\prime,X^{\prime\prime}}) \; \alpha^A_{\Phi(X),\Phi(X^\prime),\Phi(X^{\prime\prime})} \; (\Phi^{X,X^\prime} \otimes_A \id_{\Phi(X^{\prime\prime})})\\ &=\Phi^{X,X^\prime X^{\prime\prime}} \; \Phi(\alpha_{X,X^\prime,X^{\prime\prime}})\; \Phi_{XX^\prime,X^{\prime\prime}}.
\end{align*}
\end{itemize}

We provide some details here, but most  of the details will be left to the reader. Note that in the diagrams below, we will omit the $\otimes$ symbol in the nodes and arrows, and also omit parentheses in the arrows,  to make them more compact.

\smallskip

\noindent (a) The right and left $A$-module structure of $\Phi(X)=(A\otimes X)\otimes A$ are given by 
\begin{align*}
    \rho^A_{\Phi(X)}
    &:=(\id_{AX}\otimes m_A)\; \alpha_{AX, A, A},\\
    \lambda^A_{\Phi(X)}&:=((m_A\otimes \id_{X})\; \alpha^{-1}_{A, A, X}\otimes \id_A) \; \alpha^{-1}_{A, AX, A},
\end{align*}
respectively.  We leave the details for the verification of the left $A$-module condition, right $A$-module structure, and the $A$-bimodule compatibility to the reader.
     
\smallskip
     
\noindent (b.1) We obtain that $\widetilde{\Phi}_{X, X^\prime}(\rho_{\Phi(X)}\otimes \id_{\Phi(X^\prime)})=\widetilde{\Phi}_{X, X^\prime}(\id_{\Phi(X)}\otimes \lambda_{\Phi(X^\prime)})\alpha_{\Phi(X), A, \Phi(X^\prime)}$ in due to the associativity of $m_A$.  Therefore,  $\Phi_{X, X^\prime}: \Phi(X)\otimes_A\Phi(X^\prime)\to \Phi(X\otimes X^\prime)$ is a unique map  such that  $\widetilde{\Phi}_{X, X^\prime}=\Phi_{X, X^\prime} \; \pi_{\Phi(X), \Phi(X^\prime)}$. 

\smallskip

\noindent (b.2) Let us prove that $\Phi_{X, X^\prime}$ is a right $A$-module map when $\mathcal{C}$ is strict. The rest of the proof, including the non-strict case, is left to the reader. Consider the following diagram.
\[ 
{\footnotesize
\begin{array}{ccc} 
\xymatrix@R-.9pc@C+2pc{
(\Phi(X)\otimes_A\Phi(X^\prime)) A\ar @{} [drr] |{\hspace{.25in}(1)}\ar @{} [ddr] |{\hspace{-.3in}(2)}\ar[ddd]_{\rho_{\Phi(X)\otimes_A\Phi(X^\prime)}}\ar[rr]^{\Phi_{X, X^\prime}\id_{A}} &  &\Phi(X X^\prime) A\ar[ddd]^{\rho_{\Phi(X X^\prime)}}\ar @{} [ddl] |{\hspace{.3in}(3)}\\
 & \Phi(X)\Phi(X^\prime) A\ar[lu]_(.4){\pi_{\Phi(X),\Phi(X^\prime)}\id_A}\ar[ur]^(.4){\widetilde{\Phi}_{X, X^\prime}\id_A}\ar[d]_{\id_{\Phi(X)}\rho_{\Phi(X^\prime)}} & \\
 & \Phi(X)\Phi(X^\prime) \ar[dl]_{\pi_{\Phi(X), \Phi(X^\prime)}} \ar[dr]^{\widetilde{\Phi}_{X, X^\prime}} & \\
\Phi(X)\otimes_A\Phi(X^\prime)\ar[rr]_{\Phi_{X, X^\prime}}\ar @{} [urr] |{\hspace{.25in}(4)} & & \Phi(X X^\prime)
} 
\end{array}
}
\]  
We have that $(1)$ and $(4)$ commute by the definition of $\widetilde{\Phi}_{X, X^\prime}$, and $(2)$ commutes by the definition of $\rho_{\Phi(X)\otimes_A \Phi(X^\prime)}$. Moreover, $(3)$ clearly commutes. 

\smallskip

\noindent (c) We get that $\Phi_0$ is a right $A$-module map when $\mathcal{C}$ is strict because $A$ is Frobenius, and we leave the rest to the reader.

\smallskip

\noindent (d) We leave these details to the reader.

\smallskip

\noindent (e) Let us check that one of the Frobenius conditions holds for $\mathcal{C}$ strict; the rest is left to the reader. Consider the diagram below.  
\begin{equation*}
\hspace{-.1in}
\resizebox{\displaywidth}{!}{
\xymatrix@R+1pc{
    \Phi(X)\otimes_A\Phi(X^\prime X^{\prime\prime})\ar[rrr]^{\Phi_{X, X^\prime X^{\prime\prime}}}\ar[ddd]^{\id_{\Phi(X)}\otimes_A\Phi^{X^\prime, X^{\prime\prime}}}\ar @{} [drr] |{\hspace{.8in}(1)} \ar @{} [dddr] |{\hspace{.2in}(2)}\ar @{} [dddrrr] |{\hspace{-.6in}(6)}& & & \Phi(XX^\prime X^{\prime\prime})\ar[ddddd]^{\Phi^{XX^\prime, X^{\prime\prime}}}\ar[lddd]^{\widetilde{\Phi}^{XX^\prime, X^{\prime\prime}}} \\
     &\Phi(X)\Phi(X^\prime X^{\prime\prime})\ar[lu]_{\pi_{\Phi(X), \Phi(X^\prime X^{\prime\prime})}}\ar[rru]^{\widetilde{\Phi}_{X, X^\prime X^{\prime\prime}}} \ar[d]_{\id_{\Phi(X)}\Phi^{X^\prime X^{\prime\prime}}}\ar@/^5.5pc/[dd]^{\id_{\Phi(X)}\widetilde{\Phi}^{X^\prime X^{\prime\prime}}}\ar @{} [drr] |{(8)}& & \\
     & \Phi(X)(\Phi(X^\prime)\otimes_A\Phi(X^{\prime\prime}))\ar@/_.4pc/[ld]_{\pi_{\Phi(X),\Phi(X^\prime)\otimes_A\Phi(X^{\prime\prime})\hspace{.2in}}} & & \\
    \Phi(X)\otimes_A(\Phi(X^\prime)\otimes_A\Phi(X^{\prime\prime}))\ar[dd]^{{\alpha^A_{\Phi(X), \Phi(X^\prime), \Phi(X^{\prime\prime})}}^{-1}} \ar @{} [r] |{\hspace{.4in}(3)} &\Phi(X)\Phi(X^\prime)\Phi(X^{\prime\prime})\ar[u]^(.4){\id_{\Phi(X)}\pi_{\Phi(X^\prime), \Phi(X^{\prime\prime})}} \ar[d]_{\pi_{\Phi(X), \Phi(X^{\prime})}\id_{\Phi(X^{\prime\prime})}}\ar[r]^{\widetilde{\Phi}_{X, X^\prime}\id_{\Phi(X^{\prime\prime})}} & \Phi(XX^\prime)\Phi(X^{\prime\prime})\ar[rdd]^{\pi_{\Phi(XX^\prime), \Phi(X^{\prime\prime})}}& \\
    &(\Phi(X)\otimes_A\Phi(X^\prime))\Phi(X^{\prime\prime})\ar[ru]_{\hspace{.2in}\Phi_{X, X^\prime}\id_{\Phi(X^{\prime\prime})}}\ar[ld]^{\hspace{.4in}\pi_{\Phi(X)\otimes_A\Phi(X^\prime),\Phi(X^{\prime\prime})}} & & \\
   (\Phi(X)\otimes_A\Phi(X^\prime))\otimes_A\Phi(X^{\prime\prime})\ar[rrr]_{\Phi_{X, X^\prime}\otimes_A\id_{\Phi(X^{\prime\prime})}} \ar @{} [urrr] |{\hspace{1.4in}(4)}\ar @{} [uuurrr] |{\hspace{-.4in}(5)} & & & \Phi(XX^\prime)\otimes_A\Phi(X^{\prime\prime}),\ar @{} [luuuu] |{\hspace{.4in}(7)}}
}\end{equation*} 
 The diagrams $(2)$ and $(4)$ commute from \eqref{eq:tenAmap}, and  $(3)$ commutes from the definition of the associativity constraint $\alpha^A$. Moreover, $(1)$ and  $(5)$ commute from the definition of $\Phi_{\ast,\ast}$, and $(6)$ and $(7)$ commute from the definition of $\Phi^{\ast,\ast}$. Lastly, $(8)$ is the following diagram:
\[
\hspace{-.15in}{\footnotesize
\begin{array}{ccc}
\xymatrix@C+.2pc@R+.5pc{
     AXAAX^\prime X^{\prime\prime}A \ar[d]^{\id_{AXAAX^\prime} u_A\id_{X^{\prime\prime }A}}\ar[rr]^-{\id_{AX}m_A\id_{X^\prime X^{\prime\prime}A}}&  &  AXAX^\prime X^{\prime\prime}A \ar[d]^{\id_{AXAX^\prime}u_A\id_{X^{\prime\prime} A}}\ar[rr]^-{\id_{AX}\varepsilon_A\id_{X^\prime X^{\prime\prime}A}}& & AXX^\prime X^{\prime\prime}A \ar[d]^{\id_{AXX^\prime}u_A\id_{X^{\prime\prime}A}}\\
AXAAX^\prime AX^{\prime\prime}A\ar[rr]^-{\id_{AX}m_A\id_{X^\prime AX^{\prime\prime}A}}\ar[d]^{\id_{AXAAX^\prime}\Delta_A\id_{X^{\prime\prime}A}}& & AXAX^\prime AX^{\prime\prime}A\ar[d]^{\id_{AXAX^\prime}\Delta_A\id_{X^{\prime\prime}A}}\ar[rr]^{\id_{AX}\varepsilon_A\id_{X^{\prime}AX^{\prime\prime}A}} & & AXX^\prime AX^{\prime\prime}A\ar[d]^{\id_{AXX^\prime}\Delta_A\id_{X^{\prime\prime}A}}\\
AXAAX^\prime AAX^{\prime\prime}A\ar[rr]^{\id_{AX}m_A\id_{X^\prime AAX^{\prime\prime}A}}& &AXAX^\prime AAX^{\prime\prime}A\ar[rr]^{\id_{AX}\varepsilon_A\id_{X^\prime AAX^{\prime\prime}A}} & & AXX^\prime AAX^{\prime\prime} A 
}
\end{array}}
\] 
where each square commutes because the maps are applied in different slots. 
\end{proof}

\begin{remark} \label{rem:Phi}
In the theorem above we gave the free functor $\Phi: \C \to {}_A \C_A$ the structure of a Frobenius monoidal functor when the  algebra $A$ is special Frobenius. 
\begin{enumerate}
    \item Observe that $\Phi$ is not strong monoidal  if $A \not \cong \unit_\C$.
    \smallskip
    \item In the proof above, we did not need the full requirement that $A$ is special; we only used the condition that $m_A \Delta_A = \id_A$.
    \smallskip
    \item It is natural to consider connections to its (left or right) adjoint, the forgetful functor $U: {}_A \C_A \to \C$. We have that $U$ is Frobenius in the sense that its left and right adjoint are isomorphic (see, e.g., \cite[Lemma~2.1]{Shimizu2017}). It is discussed when $U$ admits a Frobenius monoidal  structure in \cite[Theorem~6.2]{BT2015}; see also \cite[Lemma~6.4]{Sz}.
\end{enumerate}
\end{remark}

In fact, we will employ the forgetful functor $U$ in the next section to study the Morita equivalence of algebras in ${}_A \C_A$.

\smallskip

\section{Morita equivalence of algebras in a category of bimodules} \label{sec:MoritaPhi}

In this section, recall that $\C$ is a fusion category,   and take $A$  a connected, special Frobenius algebra in $\C$.  Our main result is on the Morita equivalence of algebras in the monoidal category of bimodules ${}_A \C_A$,  given in Theorem~\ref{thm:Morita-updown} below. To begin, consider the following result and terminology.
In its proof, we use some auxiliary results included in the Appendix.
\begin{theorem} \label{thm:Moritapres}
Let $(\mathcal{S}, \otimes_{\mathcal{S}})$ and $(\mathcal{T}, \otimes_{\mathcal{T}})$ be fusion categories. Take a monoidal functor $\Gamma: \mathcal{S} \to \mathcal{T}$ that preserves epimorphisms and so that the natural transformation $\Gamma_{*,*}$ of $\Gamma$ is an epimorphism. If $S$ and $S'$ are Morita equivalent algebras in $\mathcal{S}$, then $\Gamma(S)$ and $\Gamma(S')$ are Morita equivalent algebras in $\mathcal{T}$.
\end{theorem}

\begin{proof}
By Proposition~\ref{prop:Morita-bimod}(a), we have bimodules \[
\begin{array}{llll}
\medskip
(\overline{P},  &\lambda_{\overline{P}}^S: S \otimes_{\mathcal{S}} \overline{P} \to \overline{P},  &\rho_{\overline{P}}^{S'}:  \overline{P} \otimes_{\mathcal{S}} S' \to \overline{P}) &\in {}_S \mathcal{S}_{S'},\\
(\overline{Q},  &\lambda_{\overline{Q}}^{S'}:  S' \otimes_{\mathcal{S}} \overline{Q} \to \overline{Q},  &\rho_{\overline{Q}}^{S}: \overline{Q} \otimes_{\mathcal{S}} S \to \overline{Q}) &\in {}_{S'} \mathcal{S}_{S},
\end{array}
\] 
equipped with isomorphisms $\overline{\tau}: \overline{P} \otimes_{S'} \overline{Q} \overset{\sim}{\to} S$ in ${}_S \mathcal{S}_{S}$ and $\overline{\mu}: \overline{Q} \otimes_{S} \overline{P} \overset{\sim}{\to} S'$ in ${}_{S'} \mathcal{S}_{S'}$. Take $$P:=\Gamma(\overline{P}), \quad \quad Q:=\Gamma(\overline{Q}).$$
By Proposition~\ref{prop:MoritaPhi-1}, we obtain the bimodules $(P, \lambda_{P}^{\Gamma(S)}, \;  \rho_{P}^{\Gamma(S')}) \in {}_{\Gamma(S)}\mathcal{T}_{\Gamma(S')}$ and 
$(Q, \lambda_{Q}^{\Gamma(S')}, \;  \rho_{Q}^{\Gamma(S)}) \in {}_{\Gamma(S')}\mathcal{T}_{\Gamma(S)}$, where 
\[
\hspace{-.05in}{\small
\begin{array}{ll}
\medskip
\lambda_{P}^{\Gamma(S)}= \Gamma(\lambda_{\overline{P}}^S) \; \Gamma_{S, \overline{P}}: \Gamma(S) \otimes_{\mathcal{T}} P \to P, &
\rho_{P}^{\Gamma(S')}= \Gamma(\rho_{\overline{P}}^{S'}) \; \Gamma_{\overline{P},S'}:  P \otimes_{\mathcal{T}} \Gamma(S') \to P,\\
\lambda_{Q}^{\Gamma(S')}= \Gamma(\lambda_{\overline{Q}}^{S'}) \; \Gamma_{S',\overline{Q}}: \Gamma(S') \otimes_{\mathcal{T}} Q \to Q, &
\rho_{Q}^{\Gamma(S)}= \Gamma(\rho_{\overline{Q}}^{S}) \; \Gamma_{\overline{Q},S}:  Q \otimes_{\mathcal{T}} \Gamma(S) \to Q.
\end{array}
}
\]

Consider the morphisms, where  $\otimes:= \otimes_{\mathcal{S}}$ below:
\[
\begin{array}{c}
\medskip
\widehat{\tau}:= \Gamma(\overline{\tau}) \; \Gamma(\pi_{\overline{P},\overline{Q}}^{S'}) \; \Gamma_{\overline{P},\overline{Q}}: P \otimes_{\mathcal{T}} Q \to \Gamma(\overline{P} \otimes
\overline{Q}) \to \Gamma(\overline{P} \otimes_{S'} \overline{Q}) \to \Gamma(S),\\
\widehat{\mu}:= \Gamma(\overline{\mu}) \; \Gamma(\pi_{\overline{Q},\overline{P}}^{S}) \; \Gamma_{\overline{Q},\overline{P}}: Q \otimes_{\mathcal{T}} P \to \Gamma(\overline{Q} \otimes
\overline{P}) \to \Gamma(\overline{Q} \otimes_{S} \overline{P}) \to \Gamma(S').
\end{array}
\]
Both $\widehat{\tau}$ and $\widehat{\mu}$ are epimorphisms (in ${}_{\Gamma(S)}\mathcal{T}_{\Gamma(S)}$ and ${}_{\Gamma(S')}\mathcal{T}_{\Gamma(S')}$, respectively) because the morphisms $\overline{\tau}$, $\overline{\mu}$, $\pi_{*,*}$ are each epic,  the natural transformation $\Gamma_{*,*}$ of $\Gamma$ is an epimorphism, and $\Gamma$ preserves epimorphisms by assumption. Moreover, the epimorphisms  $\widehat{\tau}$ and $\widehat{\mu}$ factor through epimorphisms  
\[
\begin{array}{ll}
\medskip
\tau: P \otimes_{\Gamma(S')} Q \twoheadrightarrow \Gamma(S) &\in {}_{\Gamma(S)}\mathcal{T}_{\Gamma(S)},\\
\mu: Q \otimes_{\Gamma(S)} P \twoheadrightarrow \Gamma(S') &\in {}_{\Gamma(S')}\mathcal{T}_{\Gamma(S')},
\end{array}
\]
so that 
\begin{align}\label{Tau hat - Mu hat}
\widehat{\tau} = \tau \; \pi_{P,Q}^{\Gamma(S')},\hspace{0.5in} \widehat{\mu} = \mu \; \pi_{Q,P}^{\Gamma(S)}.\end{align} 
Indeed, $$\widehat{\tau}(\rho_P^{\Gamma(S')} \otimes \id_Q) = \widehat{\tau}(\id_P \otimes \lambda_Q^{\Gamma(S')})\alpha_{P,\Gamma(S'),Q},$$ which is verified by the commutative diagram below in the strict case. The regions commute due to the monoidal  structure of $\Gamma$ and by the definitions of $\rho_P^{\Gamma(S')}$, of $\lambda_{\overline{Q}}^{S'}$, of $\widehat{\tau}$, and of $P \otimes_{\Gamma(S')} Q$. Here, $\otimes:= \otimes_{\mathcal{S}}$ in the diagram below.
\[
{\scriptsize
\xymatrix@C+1pc{
P \otimes_\mathcal{T} \Gamma(S') \otimes_\mathcal{T} Q 
\ar@/^1.5pc/[rrr]^{\id_P\; \otimes_\mathcal{T} \;\lambda_Q^{\Gamma(S')}}
\ar[r]_{\id_P \otimes_\mathcal{T} \Gamma_{S',\overline{Q}}}
\ar@/_4.2pc/[ddd]^{\rho_P^{\Gamma(S')}  \otimes_\mathcal{T}  \id_Q}
\ar[d]^{\Gamma_{\overline{P},S'}  \otimes_\mathcal{T} \id_Q}
& P \otimes_\mathcal{T} \Gamma(S' \otimes \overline{Q})
\ar[rr]_{\id_P \; \otimes_\mathcal{T} \; \Gamma(\lambda_{\overline{Q}}^{S'})}
\ar[d]^{\Gamma_{\overline{P},S' \otimes \overline{Q}}}
&& P \otimes_\mathcal{T} Q
\ar[d]_{\Gamma_{\overline{P},\overline{Q}}}
\ar@/^2.2pc/[ddd]_{\widehat{\tau}}\\
\Gamma(\overline{P} \otimes S') \otimes_\mathcal{T} Q
\ar[r]^{\Gamma_{\overline{P} \otimes S', \overline{Q}}}
\ar[dd]^{\Gamma(\rho_{\overline{P}}^{S'})  \otimes_\mathcal{T}  \id_Q}
& \Gamma(\overline{P} \otimes S' \otimes \overline{Q})
\ar[rr]^{\Gamma(\id_{\overline{P}} \; \otimes \; \lambda_{\overline{Q}}^{S'})}
\ar[dd]^{\Gamma(\rho_{\overline{P}}^{S'} \; \otimes \; \id_{\overline{Q}})}
&& \Gamma(\overline{P} \otimes \overline{Q})
\ar[d]_{\Gamma(\pi_{\overline{P}, \overline{Q}}^{S'})}\\
&&&\Gamma(\overline{P} \otimes_{S'} \overline{Q})
\ar@{=}[dl]
\ar[d]_{\Gamma(\overline{\tau})} \\
P \otimes_\mathcal{T} Q 
\ar[r]^{\Gamma_{\overline{P}, \overline{Q}}}
\ar@/_1pc/[rrr]_{\widehat{\tau}}
& \Gamma(\overline{P} \otimes \overline{Q})
\ar[r]^{\Gamma(\pi_{\overline{P}, \overline{Q}}^{S'})}
& \Gamma(\overline{P} \otimes_{S'} \overline{Q}) 
\ar[r]^{\Gamma(\overline{\tau})} 
& \Gamma(S)
}
}
\]
So the epimorphism $\tau$ exists by Definition~\ref{def:pi}. Likewise,  the epimorphism $\mu$ exists.  Finally, $\tau$ and $\mu$ satisfy diagrams $(\ast)$ and $(\ast \ast)$ in Proposition~\ref{prop:Morita-bimod}(b) by Proposition~\ref{prop:MoritaPhi-2}.
Therefore, by Proposition~\ref{prop:Morita-bimod}(b), the algebras $\Gamma(S)$ and $\Gamma(S')$ are Morita equivalent in $\mathcal{T}$.
\end{proof}

\begin{definition} \label{def:Moritapres}
We call a monoidal  functor $\Gamma: \mathcal{S} \to \mathcal{T}$ {\it Morita preserving} if it satisfies the conclusion of Theorem~\ref{thm:Moritapres}.
\end{definition}

  Consider the following notation.

\begin{notation}[$E:=E(A)$] \label{not:Sec4}
Take an algebra $A$ in a fusion category $\C$ and denote by $E:=E(A)$ the internal End object that represents the functor $${}_A \C_A \to   {\sf Set}, \quad X \mapsto {\sf Hom}_{\C_A}(A \otimes_A X, A);$$ see \cite[Section~7.9]{EGNO}. Namely, we take $M_1=M_2=A$ in \cite[(7.20)]{EGNO}. 
\end{notation}

\begin{proposition} \label{prop:EndFrob}
If $A$ is a  Frobenius algebra in $\C$, then
 $E$  is a  Frobenius algebra in ${}_A \C_A$. If, further, $A$ is  connected and special, then so is  $E$. In particular, $E$ is indecomposable.
\end{proposition}

\begin{proof}
The object structure of $E$ follows from \cite[Example~7.12.8]{EGNO} and the references within. In particular, $E$ $= {}^*\hspace{-.03in}A \otimes A$ with $A$-bimodule structure
\[
\begin{array}{rl}
\lambda_E^A =&(r_{{}^*\hspace{-.03in}A} \otimes \id_A)(\id_{{}^*\hspace{-.03in}A} \otimes {\sf ev}'_A \otimes \id_A)(\alpha_{{}^*\hspace{-.03in}A,A,A} \otimes \id_A)(\id_{{}^*\hspace{-.03in}A} \otimes m_A \otimes \id_{{}^*\hspace{-.03in}A A})\\
& \circ (\alpha_{{}^*\hspace{-.03in}A,A,A} \otimes \id_{{}^*\hspace{-.03in}A A})({\sf coev}'_A \otimes \id_{A {}^*\hspace{-.03in}A A})(l_A^{-1} \otimes \id_{{}^*\hspace{-.03in}A A})\alpha^{-1}_{A, {}^*\hspace{-.03in}A, A}
\end{array}
\] 
and $\rho_E^A = (\id_{{}^*\hspace{-.03in}A} \otimes m_A)\alpha_{ {}^*\hspace{-.03in}A, A,A}$.

On the other hand, consider the Frobenius algebra $\unit$ in $\C$. By Theorem~\ref{thm:Phi}, we then get that $\Phi(\unit) \in {\sf FrobAlg}({}_A \C_A)$. Now by Remark~\ref{rem:Frobpairing}(b), we have an isomorphism $\xi: {}^* \hspace{-.03in} A \overset{\sim}{\to} A$ in ${}_A \C$. So, we define a map $$\chi:=\xi^{-1} r_A \otimes \id_A: \Phi(\unit) = (A \otimes \unit) \otimes A \longrightarrow {}^* \hspace{-.03in} A \otimes A = E.$$
It is straight-forward to check that $\chi$ is an isomorphism of objects in ${}_A \C_A$. Since $\Phi(\unit)$ is Frobenius,  $E$ also admits the structure of a Frobenius algebra in ${}_A \C_A$.

Now suppose that $A$ is special. Then $\varepsilon_A u_A = \varphi \; \id_{\unit}$ for some nonzero $\varphi \in \kk$. So we get that $m_{\Phi(\unit)}\; \Delta_{\Phi(\unit)} = \varphi \; \id_{\Phi(\unit)}$ in this case: indeed, by Proposition~\ref{prp:monfunctor}(a,b), we have
$m_{\Phi(\unit)}\; \Delta_{\Phi(\unit)} =
\Phi(m_\unit) \; \Phi_{\unit,\unit} \; 
\Phi^{\unit,\unit} \; \Phi(\Delta_\unit)
= \varphi \; \id_{\Phi(\unit)},$ and 
$\varepsilon_{\Phi(\unit)}\; u_{\Phi(\unit)} =
\Phi^0 \; \Phi(\varepsilon_\unit)\; \Phi(u_\unit) \; \Phi_0 
=  \id_A.$
By the isomorphism $\chi$ above, and one can then rescale the multiplication of $E$ to yield that $E$ is special.
 Since $\Phi$ is a left adjoint functor of $U$, $\Hom_{_A\C_A}( \Phi(\unit), A)\cong \Hom_\C(\unit, U(A))$. Since $U(A) = A$ and $A$ is connected, we get that $\dim_\kk \Hom_{_A\C_A}(A, \Phi(\unit) )  =  \dim_\kk \Hom_{_A\C_A}( \Phi(\unit), A)  =  \dim_\kk  \Hom_\C(\unit, A)  = 1$. Therefore $E \cong \Phi(\unit)$ is connected. From Remark \ref{rmk:properties}(b), $E$ is indecomposable.
\end{proof}

By the proposition above, $E$ is an indecomposable, special Frobenius algebra in ${}_A \C_A$, when $A$ is  connected and special Frobenius.  Now recall that ${}_A \C_A$ is fusion in this case [Proposition~\ref{prop:Yam}, Remark \ref{rmk:properties}(b)]. Moreover, recall the functor $$\Phi = \Phi_A^{\C}: \C \to {}_A \C_A$$ from Theorem~\ref{thm:Phi}, and consider the following functors:
\[
\begin{array}{rll}
\medskip
\widehat{\Phi} := \Phi_E^{{}_A \C_A}: &{}_A \C_A \to {}_E({}_A \C_A)_E, &\\
\medskip
\widehat{U}: &{}_E({}_A \C_A)_E \to {}_A \C_A  &\text{(forget)},\\
U: &{}_A \C_A \to \C &\text{(forget)}.\\
\end{array}
\]

\begin{corollary} \label{cor:Moritapres}
When $A$ is a connected, special Frobenius algebra in $\C$, the functors $\Phi$, $\widehat{\Phi}$, $U$, $\widehat{U}$ are each monoidal  and  Morita preserving.
\end{corollary}

\begin{proof}
We have that $\Phi$ is monoidal  by Theorem~\ref{thm:Phi}, and $\widehat{\Phi}$ is also  monoidal  by  applying Theorem~\ref{thm:Phi} with Proposition~\ref{prop:EndFrob}. Moreover, it is straight-forward to check that $U$ is  monoidal  with the following structure: for $Y, Y' \in {}_A \C_A$, take
\begin{equation}\label{eq:Umonoidal}
\begin{array}{rl}
\medskip
U_{Y,Y'} = \pi_{Y,Y'}^A: &U(Y) \otimes U(Y') = Y \otimes Y' \to Y \otimes_A Y' = U(Y \otimes_A Y'),\\
U_0 = u_A: &\unit \to A = U(\unit_{ {}_A \C_A}).
 \end{array}
\end{equation}
For instance, the following diagram commutes due to the unit constraint on $Y$ (as a left $A$-module in $\C$) and by definition of $l_Y^A$:
\[
{\footnotesize
\xymatrix@C+6pc{
\unit \otimes U(Y) = \unit \otimes Y 
\ar[d]_{l_Y}
\ar[r]^(.45){U_0 \otimes \id_Y = u_A \otimes \id_Y} 
& A \otimes Y = U(\unit_{ {}_A \C_A}) \otimes U(Y)
\ar[d]^{U_{A,Y} = \pi_{A,Y}^A}
\ar[ld]_{\lambda_Y}\\
U(Y) = Y & A \otimes_A Y = U(A \otimes_A Y)
\ar[l]_{l_Y^A}
}
}
\]
In a similar manner, the functor $\widehat{U}$ has a  monoidal  structure.

Next, we apply Theorem~\ref{thm:Moritapres} to get that each of $\Phi$, $\widehat{\Phi}$,  $U$, $\widehat{U}$ are Morita preserving. Indeed, it is clear from Theorem~\ref{thm:Phi} that the natural transformations $\Phi_{*,*}$ and  $\widehat{\Phi}_{*,*}$ are epimorphisms. Moreover,  $\Phi$ and $\widehat{\Phi}$ are left adjoints (to $U$ and $\widehat{U}$, respectively), so they preserve epimorphisms. On the other hand, we see that the natural transformations $U_{*,*}$ and  $\widehat{U}_{*,*}$ are epimorphisms from \eqref{eq:Umonoidal}. Lastly, $U$ and $\widehat{U}$ preserve epimorphisms as they are also left adjoints (see Remark~\ref{rem:Phi}(c)). 
\end{proof}

Now we establish the main result of the section. But first we need a preliminary result on the algebra $U\;\widehat{U}\;\widehat{\Phi}\;\Phi(B)$ in $\C$ (resulting from the corollary above). 

\begin{lemma} \label{lem:Morita-updown}
For $A$  a special Frobenius algebra in $\C$ and $B \in {\sf Alg}(\C)$, we get the following statements.
\begin{enumerate}
    \item $D:=U\;\widehat{U}\;\widehat{\Phi}\;\Phi(B) \in {\sf Alg}(\C).$ Here, $D = \left(E \otimes_A ((A \otimes B) \otimes A)\right) \otimes_A E$ as an object in $\C$.
    \smallskip
    \item $D$ is isomorphic to 
    $$\hspace{.5in} T:=(U(E) \otimes B) \otimes U(E) = (E \otimes B) \otimes E$$ as objects in $\C$ via 
    $$\hspace{.5in} \theta := (\nu_E^A \otimes \id_B \otimes \mu_E^A)(\overline{\alpha}_{E,A,B}^{-1} \otimes \id_{AE})\; \overline{\alpha}_{EAB, A, E} \;(\overline{\alpha}_{E,AB,A}^{-1} \otimes_A \id_E): D \overset{\sim}{\to} T,$$
    for natural isomorphisms $\mu_E^A: A \otimes_A E \overset{\sim}{\to} E$ and $\nu_E^A: E \otimes_A A \overset{\sim}{\to} E$, and associativity constraint $\overline{\alpha}$ given in Lemma~\ref{lem:alphabar}.
     \smallskip
    \item $T$ admits the structure of an algebra in $\C$, with $m_T = \theta \; m_D \; (\theta^{-1} \otimes \theta^{-1})$ and $u_T = \theta\; u_D$, and $D \cong T$ as algebras in  $\C$.
\end{enumerate}
\end{lemma}

\begin{proof}
Part (a) follows from Corollary~\ref{cor:Moritapres} and Proposition~\ref{prp:monfunctor}. Part (b) holds because $\mu_E^A$ and $\nu_E^A$ are isomorphisms in $\C$ (see \cite[Exercise~7.8.22]{EGNO}), and $\overline{\alpha}$ is an isomorphism in $\C$ by Lemma~\ref{lem:alphabar}. Part (c) follows from parts (a) and (b).
\end{proof}

\begin{theorem} \label{thm:Morita-updown}
Take $A$  a connected, special Frobenius algebra in $\C$, and take $B, B' \in {\sf Alg}(\mathcal{C})$.  Then, $B$ and $B'$ are Morita equivalent as algebras in~$\C$ if and only if $\Phi(B)$ and $\Phi(B')$ are Morita equivalent as algebras in ${}_A \C_A$.
\end{theorem}

\begin{proof}
The forward direction holds because $\Phi$ is Morita preserving by Corollary~\ref{cor:Moritapres}.

For the converse, note that the algebras $U\;\widehat{U}\;\widehat{\Phi}\;\Phi(B)$ and $U\;\widehat{U}\;\widehat{\Phi}\;\Phi(B')$ are Morita equivalent algebras in $\C$ because $U$, $\widehat{U}$, $\widehat{\Phi}$ are each Morita preserving [Corollary~\ref{cor:Moritapres}]. So it suffices to show that $U\;\widehat{U}\;\widehat{\Phi}\;\Phi(B)$ is Morita equivalent to $B$ as algebras in $\C$, which we achieve as follows.

By Lemma~\ref{lem:Morita-updown}, $D:=U\;\widehat{U}\;\widehat{\Phi}\;\Phi(B)$ is isomorphic to $T=(U(E) \otimes B) \otimes U(E) = (E \otimes B) \otimes E$ as algebras in $\C$. So it suffices to show that $T$ is Morita equivalent to $B$ in $\C$. This holds using the methods in the proof of Proposition~\ref{prop:MoritaMat}(b). We discuss this in the strict case and leave the general case to the reader.

We have by Remark~\ref{rem:Frobpairing}(a) and Proposition~\ref{prop:EndFrob} that $E$ is a self-dual object in ${}_A \C_A$ with evaluation map $p_E = \varepsilon_E m_E: E \otimes_A E \to A$. To proceed, recall Notation~\ref{not:tilde} and define the morphism
 $$\phi:= \varepsilon_A \;\widetilde{p}_E:  E \otimes E \longrightarrow \unit.$$
Now, take  $P = E \otimes B$ with morphisms $\lambda_P^T = (\id_E \otimes m_B)(\id_{EB} \otimes \phi \otimes \id_B)$ and $\rho_P^B = (\id_E \otimes m_B)$, and take  $Q = B \otimes E$ with  morphisms $\lambda_Q^B = (m_B \otimes \id_E)$ and $\rho_Q^T = (m_B \otimes \id_E)(\id_{B} \otimes \phi \otimes \id_{BE})$. We  then obtain that $P \in {}_T \C_B$ and $Q \in {}_B \C_T$. Moreover, we have epimorphisms 
\[
\begin{array}{lrll}
\medskip
\widehat{\tau} = &\id_E \otimes m_B \otimes \id_E&: P \otimes Q \to T &\in {}_T \C_T, \\ \widehat{\mu} = &m_B(\id_B \otimes \phi \otimes \id_B)&: Q \otimes P \to B &\in {}_B \C_B,
\end{array}
\]
which factor through epimorphisms $\tau: P \otimes_B Q \to T$ and $\mu: Q \otimes_T P \to B$, respectively.
Similar to the proof of Proposition~\ref{prop:MoritaMat}(b), it is also straight-forward to check that $\tau$ and $\mu$ satisfy diagrams $(\ast)$ and $(\ast \ast)$ of Proposition~\ref{prop:Morita-bimod}(b). Thus, by Proposition~\ref{prop:Morita-bimod}, $T$ and $B$ are Morita equivalent in $\C$, as desired.
\end{proof}


\section{Algebras in pointed fusion categories} \label{sec:pointed}

We recall here the definition of a twisted group algebra in the pointed fusion category $\vecgw$ [Definition~\ref{def:ALpsi}]. We show that these algebras can be given the structure of a Frobenius algebra in  $\vecgw$ [Proposition~\ref{prop:ALpsi}], and further, that they enjoy nice properties  [Proposition~\ref{prop:Alpsi props}]. We begin with discussing pointed categories.

\begin{definition} 
A fusion category  is  called {\it pointed} if all of its simple objects are invertible, in the sense that the co/evaluation maps on simple objects are isomorphisms.
\end{definition}

The following pointed fusion category will be crucial to our work.

\begin{definition}[$\vecgw$, $\delta_g$] 
\label{def:vecgw} Take $G$ a finite group with normalized 3-cocycle $\omega \in H^3(G, \kk^\times)$. The category $\vecgw$ is  the category of $G$-graded vector spaces $V = \bigoplus_{x \in G} V_x$ with associativity constraint $\omega$ given as follows. In particular, its simple objects are  $\{\delta_g\}_{g \in G}$, where the $G$-grading is $(\delta_g)_x = \delta_{g,x}\cdot \kk$, for $g,x \in G$. 
Morphisms are $\kk$-linear maps that preserve the $G$-grading.

The monoidal structure is determined by the $G$-grading of objects $$(V \otimes W)_x = \bigoplus_{yz=x} V_y \otimes W_z,$$ the associativity constraint 
$$\alpha_{\delta_g, \delta_{g'}, \delta_{g''}} = \omega^{-1}(g, g', g'')\; {\sf \id}_{\delta_{g g' g''}}: (\delta_g \otimes \delta_{g'}) \otimes \delta_{g''} \to \delta_g \otimes (\delta_{g'} \otimes \delta_{g''}),$$
the unit object $\delta_e$,  with unit constraints $l_{\delta_g} = \omega^{-1}(e,e,g) \; \id_{\delta_g}$,   $r_{\delta_g} = \omega(g,e,e) \; \id_{\delta_g}$.

The duals of simple objects are defined as $\delta_g^* = \delta_{g^{-1}} = {}^*\delta_g$, with evaluation morphisms given by ${\sf ev}_{\delta_g}(\delta_g^* \otimes \delta_g) = \omega(g, g^{-1}, g) \delta_e$ and ${\sf ev}'_{\delta_g}(\delta_g \otimes {}^*\delta_g) = \omega^{-1}(g, g^{-1}, g) \delta_e$, and coevaluation morphisms given by ${\sf coev}_{\delta_g}(\delta_e) = \delta_g \otimes \delta_g^*$ and ${\sf coev}'_{\delta_g}(\delta_e) = {}^*\delta_g \otimes \delta_g$.
\end{definition}

\begin{remark} \label{rmk:omega}
  The associativity constraint $\alpha_{\delta_g, \delta_{g'}, \delta_{g''}}$ of $\vecgw$ given in \cite{EGNO} is defined by $\omega(g, g', g'') {\id}_{\delta_{g g' g''}}$, but we need to use $\omega^{-1}(g, g', g'') \; {\id}_{\delta_{g g' g''}}$ here in order to get that the twisted group algebra $A(L, \psi)$ presented in Definition~\ref{def:ALpsi} below is an associative algebra in $\vecgw$.
\end{remark}

Not only is $\vecgw$ a pointed fusion category, we have that every pointed fusion category is equivalent to one of this type (see \cite[Section~8.8]{ENO}).

\smallskip

 For reference in computations later, the 3-cocycle condition on $\omega$ is 
\begin{equation} \label{eq:omega}
\omega(g_1g_2, g_3, g_4)\;
\omega(g_1, g_2, g_3g_4) \; = \;
\omega(g_1, g_2, g_3) \;
\omega(g_1, g_2g_3, g_4) \;
\omega(g_2, g_3, g_4)
\end{equation}

\noindent for all $g_i \in G$.

\smallskip

Next we turn our attention to algebras in, and module categories over, $\vecgw$. To continue, consider the following terminology.

\begin{definition}[$A(L,\psi)$] \label{def:ALpsi}
Take $L$ a subgroup of $G$ so that the class of $\omega|_{L^{\times 3}}$ is trivial, and take $\psi \in C^2(L, \kk^\times)$ so that $d\psi = \omega|_{L^{\times 3}}$. We  assume that $\psi$ is normalized. We define the {\it twisted group algebra} $A(L,\psi)$ in $\vecgw$ to be $\textstyle \bigoplus_{g \in L} \delta_g$ as an object in $\vecgw$, with multiplication given by $$\delta_g \otimes \delta_{g'} \mapsto \psi(g,g') \delta_{gg'}.$$
\end{definition}

It is well-known, and we will see later in Proposition~\ref{prop:ALpsi}, that $A(L,\psi)$ is indeed an associative algebra in $\vecgw$.

For reference in computations later, note that for a 2-cocycle, say $\theta$, on a subgroup $N$ of $G$ the condition that $d \theta = \omega|_{N^{\times 3}}$   is translated as follows:
\begin{equation} \label{eq:theta}
\theta(f_1, f_2f_3)\;
\theta(f_2, f_3) \; = \;
\omega(f_1, f_2, f_3) \;
\theta(f_1f_2, f_3) \;
\theta(f_1, f_2), \qquad \text{for $f_i \in N$.}
\end{equation}

\smallskip

We show now that twisted group algebras $A(L,\psi)$ are Frobenius algebras in~$\vecgw$.

\begin{proposition} \label{prop:ALpsi} The twisted group algebra $A(L,\psi)$ admits the structure of a Frobenius algebra in $\vecgw$:  for $g, g' \in L$,  it is given by
\begin{align*}
    m_{A(L,\psi)}(\delta_g \otimes \delta_{g'}) &= \psi(g,g') \; \delta_{gg'},\\
    u_{A(L,\psi)}(\delta_e) ~&=~ \delta_e,\\
    \Delta_{A(L,\psi)}(\delta_g)~ &=~ |L|^{-1} \textstyle \bigoplus\limits_{h \in L} \psi^{-1}(gh, h^{-1}) \; [\delta_{gh} \otimes \delta_{h^{-1}}],\\
    \varepsilon_{A(L,\psi)}(\delta_g)~ &=~  \delta_{g,e} \; |L| \; \delta_e.
\end{align*}
\end{proposition}

\begin{proof}
We  have that $A(L,\psi)$ is  an algebra in $\vecgw$. Indeed, recall Remark~\ref{rmk:omega}, and for the associativity of multiplication, consider the following computation:
\begin{align*}
&m_{A(L,\psi)} ( \id \otimes m_{A(L,\psi)}  ) \; \alpha [(\delta_{g} \otimes \delta_{g'}) \otimes \delta_{g''}] \\
&= \omega^{-1}(g,g',g'') \; \psi(g',g'') \; \psi(g,g'g'') \; \delta_{gg'g''}\\
&= \psi(g,g') \; \psi(gg', g'')\; \delta_{gg'g''}\\
&= m_{A(L,\psi)} ( m_{A(L,\psi)} \otimes \id ) [(\delta_{g} \otimes \delta_{g'}) \otimes \delta_{g''}].
\end{align*} 
For the  second equation, we used \eqref{eq:theta} with $\theta = \psi$ and $(f_1, f_2, f_3) = (g, g', g'')$.
Next, it is straight-forward to check that  $u_{A(L,\psi)}$ satisfies the unit axiom.
Thus, $A(L,\psi)$ is an algebra in $\vecgw$.

Moreover, by Remark~\ref{rem:Frobpairing}(a), $A(L,\psi)$ is a Frobenius algebra in $\vecgw$ via the following formulas: 
Take
\begin{align*}
p \left( \delta_{g} \otimes \delta_{g'} \right) :=
\begin{cases}
  |L|\; \psi(g, g') \; \delta_{e}, & g g' = e    \\
  0,  & g g' \neq e, 
\end{cases}
\end{align*}
\begin{align}\label{eqn:copairing}    
q (\delta_{e}):= |L|^{-1} \textstyle \bigoplus\limits_{h \in L} \; \psi^{-1}(h, h^{-1}) \; [\delta_{h} \otimes \delta_{h^{-1}}].
\end{align}
For instance, the invariance conditions for $p$ holds as:
\begin{align*}
p \; (\id \otimes m) \; \alpha \;[(\delta_g \otimes \delta_{g'})\otimes \delta_{g''}]
&= \omega^{-1}(g, g', g'') \; p \; (\id \otimes m)[\delta_g \otimes (\delta_{g'}\otimes \delta_{g''})]\\
&= \omega^{-1}(g, g', g'')\; \psi(g', g'') \; p\; [\delta_g \otimes \delta_{g'g''}]\\
&= |L|\;\omega^{-1}(g, g', g'')\; \psi(g', g'') \; \psi(g, g'g'')\delta_{gg^\prime g^{\prime\prime},e} \; \delta_{e}\\
&= |L|\; \psi(g, g')\; \psi(gg',g'') \delta_{gg^\prime g^{\prime\prime},e} \; \delta_{e}\\
&=  \psi(g, g')\; p\;  [\delta_{gg'} \otimes \delta_{g''}]\\
&= p \; (m \otimes \id)\; [(\delta_g \otimes \delta_{g'})\otimes \delta_{g''}];
\end{align*}
here the fourth equality  holds by \eqref{eq:theta} with $\theta = \psi$ and $(f_1, f_2, f_3) = (g, g', g'')$.
We leave it to the reader to check the snake relations for $p$ and $q$.

Now the comultiplication map $\Delta$ and counit map $\varepsilon$ are given as follows (again due to Remark~\ref{rem:Frobpairing}(a)):
\begin{align*}
\Delta(\delta_g) &= (m \otimes \id_A) \; \alpha^{-1}_{A,A,A}\; (\id_A \otimes q) \; r_{A}^{-1} (\delta_g)\\
&= |L|^{-1} \textstyle \bigoplus_{h \in L} \psi^{-1}(h,h^{-1}) \; (m \otimes \id_A ) \; \alpha^{-1} [\delta_g \otimes (\delta_{h} \otimes \delta_{h^{-1}})]\\
&= |L|^{-1} \textstyle \bigoplus_{h \in L} \psi^{-1}(h,h^{-1}) \;  \omega(g,h,h^{-1})\; (m \otimes \id ) [(\delta_g \otimes \delta_{h}) \otimes \delta_{h^{-1}}]\\
&= |L|^{-1} \textstyle \bigoplus_{h \in L} \psi^{-1}(h,h^{-1}) \;  \omega(g,h,h^{-1})\; \psi(g,h)\; [\delta_{gh} \otimes \delta_{h^{-1}}]\\
&= |L|^{-1}  \textstyle \bigoplus_{h \in L} \psi^{-1}(gh,h^{-1}) \;[\delta_{gh} \otimes \delta_{h^{-1}}].
\end{align*} 
Here, the ultimate equality holds by applying \eqref{eq:theta} to $(f_1, f_2, f_3) = (g, h, h^{-1})$. Moreover,
$$ \varepsilon(\delta_g)
~=~ p \; (u \otimes \id_A)\; r_{A}^{-1} (\delta_g)
~=~ p (\delta_g \otimes \delta_e)\\
~=~ \delta_{g,e}\; |L| \; \delta_e.$$ 
Therefore, $A(L,\psi) \in {\sf FrobAlg}(\vecgw)$. 
\end{proof}

Now we discuss algebraic properties of twisted group algebras;  see Section~\ref{sec:alg-mon}.

\begin{proposition} \label{prop:Alpsi props} 
The twisted group algebra $A:=A(L, \psi)$, with structural morphisms $m, u, \Delta, \varepsilon$ given in Proposition \ref{prop:ALpsi}, possesses the following properties: 
\begin{enumerate}
    \item connected;
    \item indecomposable;
    \item special;
    \item separable. 
 \end{enumerate}
\end{proposition}

\begin{proof} 
(a) We have that $\Hom_{\vecgw}(\unit_{\vecgw},A )=\Hom_{\vecgw}(\delta_e, \oplus_{g\in L}\delta_g)=\{\delta_e\mapsto \delta_e\}$, because morphisms preserve $G$-grading. Then $\dim \Hom_{\vecgw}(\unit_{\vecgw}, A)=1$ and  $A$ is connected.

\smallskip

(b) See Remark \ref{rmk:properties}(b).

\smallskip

(c) The algebra $A$ is special because
\begin{align*}
m\Delta (\delta_{g})
&= |L|^{-1} \textstyle \bigoplus_{h \in L} \psi^{-1}(gh,h^{-1}) \; m[(\delta_{gh} \otimes \delta_{h^{-1}})]\\
 &= |L|^{-1} \textstyle \bigoplus_{h \in L} \psi^{-1}(gh,h^{-1}) \; \psi(gh,h^{-1}) \; \delta_{g} \quad = \delta_g,
\end{align*}
and $\varepsilon_{A(L,\psi)}\; u_{A(L,\psi)}(\delta_e) = |L| \delta_e =  |L|\;\id_{\unit} (\delta_e)$.

\smallskip

(d) This follows from Remark~\ref{rmk:properties}(b) and part (c) above.
\end{proof}


\section{Algebras  in group-theoretical fusion categories}\label{sc: algebras in GT}


 We define in this section the main structures of interest in this work: twisted Hecke algebras [Definition~\ref{def:tha}]. These are algebras in group-theoretical fusion categories $\mathcal{C}$ [Definition~\ref{def:gt}] that are analogous to the twisted group algebras in $\vecgw$ discussed in Section~\ref{sec:pointed}.  We establish that the twisted Hecke algebras admit the structure of a Frobenius algebra in $\mathcal{C}$ [Theorem~\ref{thm:AKL}], and further, as algebras in $\C$ we show that they are indecomposable, separable, and special [Proposition~\ref{prop: AKL props}]. We also discuss when these (Frobenius) algebras are  connected in $\C$ [Proposition~\ref{prop:non-connected}]. 
We  proceed by introducing the terminology mentioned above.

\begin{definition}[$\mathcal{C}(G,\omega, K, \beta)$]\cite[Section 8.8; Definition 8.40]{ENO} \label{def:gt}
A {\it group-theoretical fusion category} is a category of bimodules of the form 
$$\mathcal{C}(G,\omega, K, \beta):= {}_{A(K, \beta)} (\vecgw)_{A(K, \beta)}$$
for a twisted group algebra $A(K, \beta)$ in $\vecgw$.
\end{definition}

This is equivalent to the functor category ${\sf Fun}_{\vecgw}( \M(K, \beta), \M(K, \beta) )^{\text{op}}$; see \cite[Proposition 7.11.1, Definition~7.12.2, and Remark 7.12.5]{EGNO}. 
Next,  we recall a description of simple objects of  group-theoretical fusion categories.

\begin{lemma} \cite[Example~9.7.4]{EGNO} \cite[Section~5]{GelakiNaidu} \label{lem:simple-gt}
Any simple object of $\C(G,\omega,K,\beta)$ is of the form $$V_{g,\rho} = \left(\bigoplus_{f \in K,k \in T} (\delta_f \otimes \delta_g) \otimes \delta_k\right)^{\oplus n_g},$$ 
where $g\in G$ is a representative of a double coset in $K\backslash G/K$, $T$ is a set of representatives of the classes in $K/K^{g^{-1}}$ for $K^{g^{-1}}:=(K\cap g^{-1}Kg)$,  $\rho: K^{g^{-1}}\to \mathrm{GL}(V)$ is a certain irreducible projective representation, 
and $n_g=\dim V$. The $A(K,\beta)$-bimodule structure on $V_{g,\rho}$ is given by the left $A(K,\beta)$-action $m_{A(K,\beta)} \otimes \id \otimes \id$ and the compatible right $A(K,\beta)$-action is determined by the left $A(K,\beta)$-action and~$\rho$.
\qed
\end{lemma}

Now we turn our attention to algebraic structures in group-theoretical fusion categories.
 
\begin{definition}[$A^{K, \beta}(L,\psi)$] \label{def:tha} Consider the functor $\Phi: \vecgw \to \mathcal{C}(G,\omega, K, \beta)$ from Theorem~\ref{thm:Phi} in the case when $\C = \vecgw$ and $A = A(K, \beta)$. We refer to 
$$\Phi(A(L,\psi))=: A^{K, \beta}(L,\psi)$$ 
as a {\it twisted Hecke algebra} in $\mathcal{C}(G,\omega, K, \beta)$.
\end{definition}

We use this terminology because the simple objects of the group-theoretical fusion category $\mathcal{C}(G,\omega, K, \beta)$ are, in part, parameterized by $K$-double cosets in $G$ [Lemma~\ref{lem:simple-gt}], and as we  see below, the multiplication  is twisted by cocycles.

\begin{theorem} \label{thm:AKL} The twisted Hecke algebra $A^{K, \beta}(L,\psi)$ equals $$\bigoplus_{g \in L; \; f, k \in K}\; (\delta_f \otimes \delta_g) \otimes \delta_k$$  as an object in $\mathcal{C}(G,\omega, K, \beta)$.  Furthermore, for $f,f',k,k',d,d' \in K$ and $g,g' \in L$, we have the following statements.

\smallskip

\begin{enumerate}
    \item[(a)] $A^{K,\beta}(L,\psi)$ has the structure of an algebra in $\mathcal{C}(G,\omega, K, \beta)$, where
    \vspace{-.1in}
    
    \begin{align*}
    &\hspace{.25in} m_{A^{K,\beta}(L,\psi)}
    [((\delta_f \otimes \delta_g) \otimes \delta_k) 
    \; \otimes_{A(K,\beta)} \;((\delta_{f'} \otimes \delta_{g'}) \otimes \delta_{k'})]\\
    & \hspace{.5in} = \; \; \delta_{kf',e} \; \omega(fgk, f'g', k') \; \omega^{-1}(fg, k, f'g') \; \omega(k, f', g') \; \omega^{-1}(f, g, g')\\
    &\hspace{.75in} \cdot \; \beta(k, f') \; \psi(g, g')\; [(\delta_f \otimes \delta_{gg'}) \otimes \delta_{k'}],
    \end{align*}
    \begin{align*} 
    \hspace{-.45in}u_{A^{K,\beta}(L,\psi)}
    (\delta_d) ~=~ \textstyle \bigoplus_{s \in K} \beta^{-1}(ds^{-1}, s) \; [(\delta_{ds^{-1}} \otimes \delta_e) \otimes \delta_s].
    \end{align*}

\smallskip

 \item [(b)] With the above, $A^{K,\beta}(L,\psi)$ is a Frobenius algebra in $\mathcal{C}(G,\omega, K, \beta)$, where
\begin{align*}
   &\hspace{.1in} \Delta_{A^{K,\beta}(L,\psi)}
    [(\delta_f \otimes \delta_g) \otimes \delta_k] 
    \\
    & \hspace{.1in} = |K|^{-1}|L|^{-1} \textstyle \bigoplus_{h \in L; \; s \in K} \; 
     \omega(f, gh, h^{-1})\;  \omega(fgh, s, s^{-1}h^{-1}) \; \omega^{-1}(fghs, s^{-1}h^{-1}, k) \\
    & \hspace{1.1in} \cdot \; \omega^{-1}(s,s^{-1},h^{-1}) \; \psi^{-1}(gh, h^{-1}) \;  \beta^{-1}(s, s^{-1})\\
     & \hspace{1.1in} \cdot \; [((\delta_f \otimes \delta_{gh}) \otimes \delta_s)
    \; \otimes_{A(K,\beta)} \;((\delta_{s^{-1}} \otimes \delta_{h^{-1}}) \otimes \delta_{k})], 
    \end{align*}
    \begin{align*}
    \hspace{-.85in} \varepsilon_{A^{K,\beta}(L,\psi)}[(\delta_f \otimes \delta_g) \otimes \delta_k] ~=~ 
\delta_{g,e} \; |L|\; \beta(f,k) \;  \delta_{fk}.
    \end{align*}
\end{enumerate}
\end{theorem}

\begin{proof}
The definition of the functor $\Phi$ gives us \[A^{K,\beta}(L,\psi):= \Phi(A(L,\psi))=(A(K,\beta)\otimes A(L,\psi))\otimes A(K,\beta),\]
which corresponds to the object
$$\bigoplus_{g \in L; \; f, k \in K}\; (\delta_f \otimes \delta_g) \otimes \delta_k$$
in the category $\mathcal{C}(G,\omega, K, \beta)$. Throughout this proof, we will fix the notation 
$$A := A(K, \beta) \qquad \text{and} \quad B := A(L, \psi)$$ for simplicity. Recall that $\Phi(f) = \id_A \otimes f\otimes \id_A$ for any morphism $f$ in $\vecgw$.

\smallskip

(a) Since the functor $\Phi$ is monoidal  [Theorem~\ref{thm:Phi}], we have by Proposition~\ref{prp:monfunctor}(a) that  $A^{K,\beta}(L,\psi)=\Phi(B)$ is an algebra in $\C(G,\omega,K,\beta)$, with multiplication and unit maps given by $\Phi(m_B)\Phi_{B,B}$ and $ \Phi(u_B)\Phi_0$, respectively. Here, the monoidal structure of $\Phi$ is defined in Theorem \ref{thm:Phi}, and in particular, the morphism $\Phi_{B,B}$ is given by means of the lift $\widetilde{\Phi}_{B,B}$ [Notation~\ref{not:tilde}].

Note that 
\begin{align*}
&\widetilde{\Phi}_{B,B}[((\delta_f\otimes \delta_g)\otimes \delta_k) \otimes(\delta_{f'}\otimes \delta_{g'})\otimes \delta_{k'})]\\
&= (\alpha_{A,B,B} \otimes \id_A) \; (\id_{A,B} \otimes l_{B} \otimes \id_A) \; (\id_{A,B} \otimes   \varepsilon_A m_A \otimes \id_{B,A})\\
&\hspace{0.3in}(\id_{A,B} \otimes \alpha^{-1}_{A,A,B} \otimes \id_A)  \; (\alpha_{AB,A,AB} \otimes \id_A)\; \alpha^{-1}_{ABA,AB,A}\\
&\hspace{0.3in}[((\delta_f\otimes \delta_g)\otimes \delta_k) \otimes ((\delta_{f'}\otimes \delta_{g'})\otimes \delta_{k'})]\\
&= \omega(fgk, f'g', k') \; \omega^{-1}(fg, k, f'g')\; \omega(k, f', g')\\
&\hspace{0.3in}(\alpha_{A,B,B} \otimes \id_A) \; (\id_{A,B} \otimes l_{B} \otimes \id_A) \; (\id_{A,B} \otimes   \varepsilon_A m_A \otimes \id_{B,A}) \\
&\hspace{0.3in}[((\delta_f\otimes \delta_g)\otimes (\delta_k\otimes (\delta_{f'}\otimes \delta_{g'})))\otimes \delta_{k'}]\\
& = \delta_{kf',e} \; \omega(fgk, f'g', k') \; \omega^{-1}(fg, k, f'g') \; \omega(k, f', g')\; \beta(k,f')\\
&\hspace{0.3in}( \alpha_{A,B,B} \otimes \id_A)(\id_{A,B} \otimes l_{B} \otimes \id_A)\\
&\hspace{0.3in} [((\delta_f\otimes \delta_g)\otimes (\delta_e\otimes \delta_{g'}))\otimes \delta_{k'}]\\
& = \delta_{kf',e} \; \omega(fgk, f'g', k') \; \omega^{-1}(fg, k, f'g') \; \omega(k, f', g') \;\beta(k,f')\; \omega^{-1}(f, g, g')\\
&\hspace{0.3in}(\delta_f\otimes (\delta_g\otimes \delta_{g'}))\otimes \delta_{k'}.
\end{align*}
Therefore, the multiplication of $A^{K,\beta}(L,\psi)$ is given by
\begin{align*}
&m_{A^{K,\beta}(L,\psi)}[((\delta_f \otimes \delta_g) \otimes \delta_k)     \; \otimes_{A} \;((\delta_{f'} \otimes \delta_{g'}) \otimes \delta_{k'})]\\
&=\delta_{kf',e} \; \omega(fgk, f'g', k') \; \omega^{-1}(fg, k, f'g') \; \omega(k, f', g') \; \omega^{-1}(f, g, g')\\
&\hspace{.75in} \beta(k, f') \; \psi(g, g')\; [(\delta_f \otimes \delta_{gg'}) \otimes \delta_{k'}].
\end{align*}
Here, we use Proposition~\ref{prop:ALpsi} for the multiplication and counit of $A$, and the monoidal  structure of $\vecgw$ is given in Definition~\ref{def:vecgw}.

On the other hand, by using the definition of $\Phi_0$ from Theorem \ref{thm:Phi} we get that the unit of $A^{K,\beta}(L,\psi)$ is given by
\begin{align*}
  &u_{A^{K,\beta}(L,\psi)}(\delta_d)
  ~=~\Phi(u_B) \; \Phi_0(\delta_d)
  ~=~\Phi(u_B) \; (r_A^{-1} \otimes \id_A) \; \Delta_A(\delta_d)\\
  &= \Phi(u_B) \; (r_A^{-1} \otimes \id_A) \textstyle \bigoplus\limits_{s \in K} \beta^{-1}(s^{-1}, s)\; \beta(d,s) \; \omega(s, s^{-1}, s) \; \omega(d, s, s^{-1}) \; [\delta_{ds} \otimes \delta_{s^{-1}}]\\
  &=  \textstyle \bigoplus\limits_{s \in K}  \beta^{-1}(s^{-1}, s)\; \beta(d,s) \; \omega(s, s^{-1}, s) \; \omega(d, s, s^{-1}) \; [(\delta_{ds} \otimes \delta_e) \otimes \delta_{s^{-1}}]\\
  &= \textstyle \bigoplus\limits_{s \in K}  \beta^{-1}(s, s^{-1})\; \beta(d,s^{-1}) \; \omega(s^{-1}, s, s^{-1}) \; \omega(d, s^{-1}, s) \; [(\delta_{ds^{-1}} \otimes \delta_e) \otimes \delta_s]\\
  &= \textstyle \bigoplus\limits_{s \in K}  \beta^{-1}(s, s^{-1})\; \beta(d,s^{-1}) \; \omega(ds, s, s^{-1}) \; [(\delta_{ds^{-1}} \otimes \delta_e) \otimes \delta_s]\\
   &= \textstyle \bigoplus\limits_{s \in K}  \beta^{-1}(ds^{-1}, s) \; [(\delta_{ds^{-1}} \otimes \delta_e) \otimes \delta_s].
\end{align*}

\noindent For the penultimate equation we used \eqref{eq:omega} with $(g_1, g_2, g_3, g_4) = (d, s^{-1}, s, s^{-1})$, and we used \eqref{eq:theta} with $\theta = \psi$ and $(f_1, f_2, f_3) = (ds^{-1}, s, s^{-1})$ for the last equation. Moreover, we use Proposition~\ref{prop:ALpsi} for the comultiplication of $A$, and again the monoidal structure of $\vecgw$ is described in Definition~\ref{def:vecgw}.

\smallskip

(b) Since the functor $\Phi$ is Frobenius monoidal  (see Theorem \ref{thm:Phi}) and $B$ is a Frobenius algebra in $\vecgw$ (see Proposition \ref{prop:ALpsi}),  $A^{K,\beta}(L,\psi) = \Phi (B)$ is a Frobenius algebra in $\mathcal{C}(G,\omega, K, \beta)$ by Proposition \ref{prp:monfunctor}(c). Moreover, the comultiplication and counit of $A^{K,\beta}(L,\psi)$ determined by $\Phi$ are $\Phi^{B,B}\Phi(\Delta_{B})$ and  $\Phi^0\Phi(\varepsilon_{B})$, respectively.
Recall that the comonoidal  structure $\Phi^{\ast,\ast}$ and $\Phi^0$ of $\Phi$ is described in Theorem~\ref{thm:Phi}, and the structure of $A$ and $B$ are given in Proposition~\ref{prop:ALpsi}. 
Now,
\begin{equation} \label{eqref:DeltaPhi}
\begin{array}{l}

\medskip

\Phi^{B,B} \; \Phi(\Delta_B) \; [(\delta_f \otimes \delta_g) \otimes \delta_k]\\ 

= |L|^{-1} \textstyle \bigoplus\limits_{h \in L} \psi^{-1}(gh, h^{-1}) \;\Phi^{B,B} \; [(\delta_f\otimes(\delta_{gh} \otimes \delta_{h^{-1}})) \otimes \delta_k].    
\end{array}
\end{equation}
Moreover,  the lift $\widetilde{\Phi}^{B,B}$ of $\Phi^{B,B}$ on $(\delta_f\otimes(\delta_{gh} \otimes \delta_{h^{-1}})) \otimes \delta_k$ is given as follows:
    \begin{align*}
    &\widetilde{\Phi}^{B,B} [(\delta_f\otimes(\delta_{gh} \otimes \delta_{h^{-1}})) \otimes \delta_k]\\
    &  = \alpha_{ABA,AB,A} \; (\alpha^{-1}_{AB,A,AB} \otimes \id_A) \; 
    (\id_{A,B} \otimes \alpha_{A,A,B} \otimes \id_A)\; (\id_{A,B} \otimes \Delta_A \otimes \id_{B,A}) \\ & \hspace{.2in} (\id_{A,B} \otimes u_A \; l_{B}^{-1} \otimes \id_A) \; (\alpha^{-1}_{A,B,B} \otimes \id_A)\;  [(\delta_f\otimes(\delta_{gh} \otimes \delta_{h^{-1}})) \otimes \delta_k]
    \\
    & = \alpha_{ABA,AB,A} \; (\alpha^{-1}_{AB,A,AB} \otimes \id_A) \; 
    (\id_{A,B} \otimes \alpha_{A,A,B} \otimes \id_A) (\id_{A,B} \otimes \Delta_A \otimes \id_{B,A})\\ & \hspace{.2in} \omega(f, gh, h^{-1}) [((\delta_f\otimes\delta_{gh}) \otimes (\delta_e\otimes\delta_{h^{-1}})) \otimes \delta_k]
    \\
    &  = \textstyle \bigoplus\limits_{s\in K} \; \omega(f, gh, h^{-1}) \; \omega(s, s^{-1}, s)\; \beta^{-1}(s^{-1},s)\; 
    \alpha_{ABA,AB,A}\; (\alpha^{-1}_{AB,A,AB} \otimes \id_A) \\
    & \hspace{.4in}
    (\id_{A,B} \otimes \alpha_{A,A,B} \otimes \id_A)  [((\delta_f\otimes\delta_{gh}) \otimes ((\delta_s\otimes\delta_{s^{-1}})\otimes\delta_{h^{-1}})) \otimes \delta_k]
    \\
    & = \textstyle \bigoplus\limits_{s\in K}\; \omega(f, gh, h^{-1})\;\omega(s, s^{-1}, s)\; \omega^{-1}(s, s^{-1}, h^{-1})\; \omega(fgh, s, s^{-1}h^{-1})\\   & \hspace{.4in} \omega^{-1}(fghs,s^{-1}h^{-1},k) \;   \beta^{-1}(s^{-1},s) \; [((\delta_f\otimes\delta_{gh}) \otimes \delta_s)\otimes((\delta_{s^{-1}}\otimes\delta_{h^{-1}}) \otimes \delta_k)],\\
    & = \textstyle \bigoplus\limits_{s\in K}\; \omega(f, gh, h^{-1})\; \omega(fgh, s, s^{-1}h^{-1}) \; \omega^{-1}(fghs,s^{-1}h^{-1},k)\\   & \hspace{.4in} \omega^{-1}(s, s^{-1}, h^{-1}) \;   \beta^{-1}(s,s^{-1}) \; [((\delta_f\otimes\delta_{gh}) \otimes \delta_s)\otimes((\delta_{s^{-1}}\otimes\delta_{h^{-1}}) \otimes \delta_k)];
\end{align*}
here, we used \eqref{eq:theta} with $\theta = \beta$ and $(f_1, f_2, f_3) = (s, s^{-1}, s)$ for the last equation. Together, with \eqref{eqref:DeltaPhi}, we can normalize $\Phi^{B,B} \Phi_B(\Delta_B)$ by multiplying by $|K|^{-1}$ to get the desired formula for $\Delta_{A^{K,\beta}(L,\psi)}$.

On the other hand,
\begin{align*}
& \varepsilon_{A^{K,\beta}(L,\psi)}[(\delta_f \otimes \delta_g) \otimes \delta_k]\\
    &~=~ \Phi^0\Phi(\varepsilon_B) [(\delta_f \otimes \delta_g) \otimes \delta_k] \\& 
    ~=~ \delta_{g,e}\; |L|  \; m_A \; (r_A\otimes \id_A)[(\delta_f \otimes \delta_e) \otimes \delta_k]
    \\& 
    ~=~ \delta_{g,e} \; |L| \; \beta(f,k)\;\delta_{fk}.
    \end{align*}
    
    \vspace{-.2in}
\end{proof}

\begin{remark} \label{rem:tha-triv}
Taking the forgetful functor $U: \mathcal{C}(G, \omega, \langle e \rangle, 1)  \to \vecgw$, observe that $U(A^{\langle e \rangle,1}(L,\psi)) \cong A(L,\psi)$ as algebras in $\vecgw$.
\end{remark}

Next, we discuss algebraic properties of twisted Hecke algebras;  see Section~\ref{sec:alg-mon}.

\begin{proposition} \label{prop: AKL props} 
The twisted Hecke algebra $A^{K, \beta}(L, \psi)$, with structural mor- \linebreak phisms $m, u, \Delta, \varepsilon$ given in Theorem \ref{thm:AKL}, possesses the following properties: 
\begin{enumerate}
     \item indecomposable;
    \item  special; and
    \item separable.
 \end{enumerate}
\end{proposition}

\begin{proof}
(a)  By way of contradiction, suppose that $A^{K, \beta}(L, \psi) = A_1 \oplus A_2$ is a decomposable algebra. Then, $A_1$ contains as a summand a simple object $V_{g,\rho}$ from Lemma~\ref{lem:simple-gt}. In that result, we can take $f = k= e$ to get that $(\delta_{e} \otimes \delta_{g}) \otimes \delta_{e}$ is a summand of $A_1$ for some $g \in L$. Since $A_1$ is closed under multiplication, $m[((\delta_{e} \otimes \delta_{g}) \otimes \delta_{e}) \otimes_{A(L,\psi)} (\delta_{e} \otimes \delta_{g}) \otimes \delta_{e})]$ is a summand of $A_1$. So, we get by Theorem~\ref{thm:AKL}(a), and by rescaling, that $(\delta_{e} \otimes \delta_{g^2}) \otimes \delta_{e}$ is a summand of $A_1$. Repeating this process, we obtain that $(\delta_{e} \otimes \delta_{e}) \otimes \delta_{e}$ is a summand of $A_1$ (as the element $g$ has finite order in $L$). Likewise, $A_2$ contains as a summand a simple object $V_{g',\rho'}$ from Lemma 6.2,  and we obtain that $(\delta_{e} \otimes \delta_{g'}) \otimes \delta_{e}$ is a summand of $A_2$ for some $g' \in L$ as a consequence. Arguing as above, $(\delta_{e} \otimes \delta_{e}) \otimes \delta_{e}$ is also a summand of $A_2$, which contradicts $A_1 \cap A_2 = (0)$. Therefore, $A^{K, \beta}(L, \psi)$ is an indecomposable algebra in $\C(G,\omega,K,\beta)$.

\smallskip

(b) To verify the special property, we compute:
\begin{align*}
    &m_{A^{K, \beta}(L, \psi)}\Delta_{A^{K, \beta}(L, \psi)}
    [(\delta_f \otimes \delta_g) \otimes \delta_k]\\ 
      & =  |K|^{-1}\; |L|^{-1}\;\textstyle \bigoplus\limits_{h \in L; \; s \in K} \; 
     \omega(f, gh, h^{-1})\;  \omega(fgh, s, s^{-1}h^{-1})\; \omega^{-1}(fghs, s^{-1}h^{-1}, k) \;\\ 
     & \hspace{1.3in}  \cdot \; \omega^{-1}(s,s^{-1},h^{-1}) \; \psi^{-1}(gh, h^{-1})\; \beta^{-1}(s, s^{-1})\; \\ 
     & \hspace{1.3in} \cdot 
     \; m_{A^{K, \beta}(L, \psi)}[((\delta_f \otimes \delta_{gh}) \otimes \delta_s)
    \; \otimes_{A(K,\beta)} \;((\delta_{s^{-1}} \otimes \delta_{h^{-1}}) \otimes \delta_{k})]\\
         & =  |K|^{-1}\; |L|^{-1}\;\textstyle \bigoplus\limits_{h \in L; \; s \in K} \; 
     \omega(f, gh, h^{-1})\;  \omega(fgh, s, s^{-1}h^{-1})\; \omega^{-1}(fghs, s^{-1}h^{-1}, k) \;\\ 
     & \hspace{1.3in}  \cdot \; \omega^{-1}(s,s^{-1},h^{-1}) \; \psi^{-1}(gh, h^{-1})\; \beta^{-1}(s, s^{-1})\; \\ 
     & \hspace{1.3in}  \cdot \; \omega(fghs, s^{-1}h^{-1},k) \; \omega^{-1}(fgh,s,s^{-1}h^{-1})\;\omega(s,s^{-1},h^{-1})\\
     & \hspace{1.3in}  \cdot \; \omega^{-1}(f,gh,h^{-1})\; \psi(gh,h^{-1})\; \beta(s,s^{-1})
     \; ((\delta_f \otimes \delta_{g}) \otimes \delta_{k})\\
    & =  |K|^{-1} \; |L|^{-1}\; \textstyle \bigoplus\limits_{h \in L; \; s \in K} \; (\delta_f \otimes \delta_g) \otimes \delta_k\\
    & =  (\delta_f \otimes \delta_g) \otimes \delta_k,
    \end{align*}
   and for $A(K,\beta) = \bigoplus_{d \in K} \delta_d$, we get
\begin{align*}
&\textstyle \varepsilon_{A^{K,\beta}(L,\psi)}\; u_{A^{K,\beta}(L,\psi)}\left(\delta_d\right)\\
&\quad~=~ \textstyle\varepsilon_{A^{K,\beta}(L,\psi)}\left(\bigoplus_{s \in K} \beta^{-1}(ds^{-1},s) [(\delta_{ds^{-1}} \otimes \delta_e) \otimes \delta_s]\right)\\
&\quad~=~ \textstyle\bigoplus_{s \in K} \beta^{-1}(ds^{-1},s) \; |L|\; \beta(ds^{-1},s)\; \delta_d\\
&\quad~=~ \textstyle |K|\; |L| \; \delta_d.
\end{align*}
Therefore, 
$$m_{A^{K,\beta}(L,\psi)}\; \Delta_{A^{K,\beta}(L,\psi)} =  \text{id}_{A^{K,\beta}(L,\psi)}, \quad  \varepsilon_{A^{K,\beta}(L,\psi)}\; u_{A^{K,\beta}(L,\psi)} = |K| \; |L| \; \text{id}_{A(K,\beta)}.$$ 

\smallskip

(c) This follows from Remark~\ref{rmk:properties}(b)  and part (b) above.
\end{proof}

Now we examine the connected property of $A^{K,\beta}(L,\psi)$.

\begin{proposition} \label{prop:non-connected}
For the twisted Hecke algebra $A^{K, \beta}(L, \psi)$, it holds that
$$\dim_{\kk}(\Hom_{\C(G,\omega,K,\beta)} (A(K,\beta),  A^{K,\beta}(L,\psi)))=|K \cap L|.$$
As a consequence, $A^{K, \beta}(L, \psi)$ is connected precisely when  $|K \cap L| = 1$.
\end{proposition}

\begin{proof}
 Take $A:=A(K,\beta)$, $B:=A(L,\psi)$, and $\C:= \vecgw$. Recall that the free functor $\Phi$ from Theorem~\ref{thm:Phi} is left adjoint to the forgetful functor $U: {}_A \mathcal{C}_A \to \mathcal{C}$ [Remark~\ref{rem:Phi}(c)]. So, $\Hom_{{}_A \mathcal{C}_A}(\Phi(B), A) \cong \Hom_{\mathcal{C}}(B,U(A))$. Since $\Phi(B) = A^{K, \beta}(L, \psi)$ and ${}_A \mathcal{C}_A = \C(G,\omega,K,\beta)$, we obtain that
$$\dim_\kk (\Hom_{\C(G,\omega,K,\beta)} (A,  A^{K,\beta}(L,\psi))) = 
\dim_\kk (\Hom_{\mathcal C} (B,A)) = |K \cap L|.$$

\vspace{-.2in}

\end{proof}

 Recall that in the special case when $K = \langle e \rangle$, $\beta = 1$, the twisted Hecke algebra $A^{K,\beta}(L,\psi)$ is, via Remark~\ref{rem:tha-triv}, the twisted group algebra $A(L,\psi)$. Here,  $\dim_{\kk}\Hom_{_A\mathcal C_A} (A,  A^{K,\beta}(L,\psi))=  1$,
 so $A(L,\psi)$ is connected.  This recovers Proposition~\ref{prop:Alpsi props}(a).


\section{Representation theory of group-theoretical fusion categories} \label{sec:repgtfc}

We provide in this section a classification of indecomposable semisimple representations of group-theoretical fusion categories in terms of the twisted Hecke algebras defined  and studied in Section~\ref{sc: algebras in GT}; see Proposition~\ref{prop:indssCA} and Theorem~\ref{thm:modtha-new} below. This result is analogous to Ostrik and Natale's   classification of indecomposable semisimple representations of pointed fusion categories  in terms of  twisted group algebras (studied in Section~\ref{sec:pointed}) \cite{OstrikIMRN2003, Natale2016}; see Theorem~\ref{thm:OstNat-new} below.

To begin, recall the notation from Sections~\ref{sec:pointed} and~\ref{sc: algebras in GT}, and consider the following notation.

\begin{notation}[${}^x \hspace{-.02in} s$,  ${}^x \hspace{-.01in} S$, $\psi^x$, $\Omega_x$, $\M(L,\psi)$, $\mathcal{M}^{K,\beta}(L,\psi)$] \label{not:conj-new} \textcolor{white}{.}
\begin{itemize}
\item We write ${}^x \hspace{-.02in} s := xsx^{-1}$ and ${}^x \hspace{-.01in} S:=\{{}^xs \colon s \in S\}$, for  $x \in G$ and any set $S$.
\smallskip
\item Take a 2-cochain $\psi$ on a subgroup $L$ of $G$ and an element $x \in G$. The 2-cochain $\psi^x$ on ${}^x \hspace{-.01in} L$ is defined by $\psi^x(h_1,h_2) = \psi({}^x \hspace{-.01in}  h_1, {}^x  \hspace{-.01in}  h_2)$ for $h_1,h_2 \in L$.

\smallskip
\item For $x \in G$, define the  2-cocycle
$\Omega_x\colon G \times G \to \kk^\times$  by
$$
\Omega_x(h_1,h_2) = \frac{\omega({}^x \hspace{-.01in}  h_1, {}^x  \hspace{-.01in} h_2, x) \; \omega(x, h_1, h_2)}{\omega({}^x h_1, x, h_2)}.$$
\item Let $\M(L,\psi)$ denote the left $\vecgw$-module category consisting of right
\linebreak $A(L,\psi)$-modules in $\vecgw$.
\smallskip
\item Let $\mathcal{M}^{K,\beta}(L,\psi)$  denote the left $\C(G,\omega,K,\beta)$-module category consisting of right $A^{K,\beta}(L,\psi)$-modules in $\C(G,\omega,K,\beta)$.
\end{itemize}
\end{notation}

Next, we borrow a condition from \cite{Natale2016}.

\begin{definition}[$\mathcal{P}(G, \omega)$] 
Let $L$, $L'$ be subgroups of $G$. Take $\psi \in C^2(L,\kk^\times)$ with $d\psi = \omega|_L$, and take $\psi' \in C^2(L',\kk^\times)$ with $d\psi' = \omega|_{L'}$. We say that the pairs $(L, \psi)$ and $(L', \psi')$ are {\it conjugate} if there exists an element 
$x \in G$ so that 
\begin{enumerate}
\item $L = {}^x \hspace{-.02in} L'$, and 
\item the class of the 2-cocycle $\psi'^{-1}\; \psi^x\;\Omega_x|_{L' \times L'}$ is trivial in $H^2(L', \kk^\times)$.
\end{enumerate}
We denote by $\mathcal{P}(G, \omega)$ the set of conjugacy classes of pairs $(L,\psi)$ as 
above.
\end{definition}

Now consider the classification result for representations of pointed fusion categories mentioned above.

\begin{theorem} \label{thm:OstNat-new} \cite[Example~2.1]{OstrikIMRN2003} \cite[Example~9.7.2]{EGNO} \cite{Natale2016}  \textcolor{white}{.}

\begin{enumerate}
\item We have that $\M(L, \psi)$ and $\M(L', \psi')$ are equivalent as ${\sf Vec}_G^\omega$-module categories if and only if $(L,\psi) = (L', \psi')$ in $\mathcal{P}(G,\omega)$.

\smallskip

\item Every indecomposable left module category over  ${\sf Vec}_G^\omega$ is equivalent to one of the form $\M(L,\psi)$, as left $\vecgw$-module categories.
\qed
\end{enumerate}
\end{theorem}

This brings us to the main result of this section, and of this article.

\begin{theorem} \label{thm:modtha-new}
We have the following statements.
\begin{enumerate}
\item $\M^{K,\beta}(L, \psi)$ and $\M^{K,\beta}(L', \psi')$ are equivalent as $\C(G,\omega,K,\beta)$-module categories if and only if $(L,\psi) = (L', \psi')$ in $\mathcal{P}(G,\omega)$.

\smallskip

\item Every indecomposable left module category  over  $\mathcal{C}(G,\omega,K,\beta)$ is equivalent to one of the form $\M^{K,\beta}(L,\psi)$, as left $\mathcal{C}(G,\omega,K,\beta)$-module categories. 
\end{enumerate}
\end{theorem}

\begin{proof}
(a) By Theorem~\ref{thm:OstNat-new}, we need to show that $\M(L, \psi)$ and $\M(L', \psi')$ are equivalent as ${\sf Vec}_G^\omega$-module categories if and only if $\M^{K,\beta}(L, \psi)$ and $\M^{K,\beta}(L', \psi')$ are equivalent as $\C(G,\omega,K,\beta)$-module categories. But this holds by using Theorem~\ref{thm:Morita-updown}, with Propositions~\ref{prop:ALpsi} and~\ref{prop:Alpsi props}(a,c), applied to $\C = \vecgw$, $A = A(K,\beta)$, $B = A(L,\psi)$, and $B' = A(L', \psi')$. 

\smallskip

(b) For a fusion category $\mathcal{D}$, let ${\sf Indec}({\sf Mod}(\mathcal{D}))$ denote a set of equivalence class representatives of indecomposable left $\mathcal{D}$-module categories, and let $[\mathcal{M}]$ be the class of $\mathcal{D}$-module categories equivalent to $\mathcal{M}$ (as left $\mathcal{D}$-module categories).

Now by Theorem~\ref{thm:OstNat-new} and \cite[Sections~3 and~4]{MugerI} (see also \cite[Theorem~7.12.11]{EGNO}), there is a 1-to-1 correspondence between the finite sets, 
$${\sf Indec}({\sf Mod}(\mathcal{C}(G,\omega,K,\beta))) \quad \quad \text{ and } \quad \quad \mathcal{P}(G,\omega);$$
namely, both of these sets are in bijection with ${\sf Indec}({\sf Mod}(\vecgw))$. 
On the other hand, since  $A^{K,\beta}(L,\psi)$ is an indecomposable and separable algebra in $\mathcal{C}(G,\omega,K,\beta)$ [Proposition~\ref{prop: AKL props}], the finite collection $$\{[\mathcal{M}^{K,\beta}(L,\psi)]\}_{(L,\psi) \in \mathcal{P}(G,\omega)}$$ consists of equivalence classes of indecomposable  left $\mathcal{C}(G,\omega,K,\beta)$-module categories [Proposition~\ref{prop:indssCA}]. (Indeed, indecomposability is preserved under module category equivalence.) Moreover, by (a), this collection  is also in bijection with the finite set $\mathcal{P}(G,\omega)$. Therefore, as finite sets, $${\sf Indec}({\sf Mod}(\mathcal{C}(G,\omega,K,\beta))) = \{[\mathcal{M}^{K,\beta}(L,\psi)]\}_{(L,\psi) \in \mathcal{P}(G,\omega)},$$ and this verifies part~(b).
\end{proof}

Finally, we compare our work with recent work of P. Etingof, R. Kinser, and the last author in \cite{EKW}.

\begin{remark} \label{rem:EKW}
Morita equivalence class representatives of indecomposable, separable algebras in group-theoretical fusion categories $\C$  were used in the recent study of tensor algebras in $\C$; see \cite[Theorem~3.11 and Section~5]{EKW}.  (Note that a `separable algebra' here is the same as a `semisimple algebra' in \cite{EKW}  as we are working over an algebraically closed field.) Now by Theorem~\ref{thm:modtha-new}, our construction of the twisted Hecke algebras in $\C$ serve  as the base algebras of tensor algebras in $\C$, up to the notion of equivalence given in \cite[Definition~3.4]{EKW}.
\end{remark}

\begin{example} \label{ex:EKW}
Continuing the remark above, let ${\sf Rep}(H_8)$ be the category of finite-dimensional representations of the Kac-Paljutkin Hopf algebra, which is a group-theoretical fusion category $\C(D_8, \omega, \mathbb{Z}_2, 1)$; see \cite[Example~5.3 and Section~5.3]{EKW} for more details. A collection of Morita equivalence class representatives of indecomposable, separable algebras (or, up to equivalence, of base algebras of the tensor algebras) in ${\sf Rep}(H_8)$ is given in \cite[Theorem~5.23]{EKW}. The correspondence of those six algebras with the conjugacy classes of pairs $(L,\psi)$ is presented in \cite[Proposition~5.26]{EKW}. Thus, we can replace the algebras in \cite[Theorem~5.23]{EKW} corresponding to such pairs $(L,\psi)$ with the twisted Hecke algebras $A^{\mathbb{Z}_2, 1}(L,\psi)$ featured here. The advantage is that the six algebras of \cite[Theorem~5.23]{EKW} were found via ad-hoc methods \cite[Remark~5.28]{EKW}, whereas our construction provides a uniform collection of Morita equivalence classes representatives of algebras in ${\sf Rep}(H_8)$.
\end{example}


\section*{Acknowledgements.} The authors thank C{\'e}sar Galindo, Ryan Kinser, Victor Ostrik, and Harshit Yadav for insightful comments on a preliminary version of this article. We especially thank the anonymous referee for their detailed comments, which greatly improved the quality of our manuscript. This work began at the Women in Noncommutative Algebra and Representation Theory (WINART2) workshop, held at the University of Leeds in May 2019. We thank
the University of Leeds' administration and staff for their hospitality and productive atmosphere. 

Y. Morales was partially supported by the London Mathematical Society, workshop grant \#WS-1718-03. M. M\"uller was partially supported by  London Mathematical Society, workshop grant \#WS-1718-03 and by Universidade Federal de Vi\c cosa - Campus Florestal. J. Plavnik gratefully acknowledges the support of Indiana University, Bloomington, through a Provost's Travel Award for Women in Science. A. Ros Camacho was supported by the NWO Veni grant 639.031.758, Utrecht University and Cardiff University. A. Tabiri was supported by the Schlumberger Foundation Faculty for the Future Fellowship, AIMS-Google AI Postdoctoral Fellowship and AIMS-Ghana. C. Walton was supported by a research fellowship from the Alfred P. Sloan foundation. J. Plavnik and C. Walton were also supported by the U.S. NSF with research grants  DMS-1802503/1917319, and DMS-1903192/2100756, respectively.

\medskip

\noindent {\it Data Availability:} Data sharing is not applicable to this article as no datasets were generated or analyzed during the current study.

\medskip

\noindent {\it Ethical Statement/Conflict of Interest:} There are no conflicts of interest for this work.


\appendix

\section{Remainder of the proof of Theorem~\ref{thm:Moritapres}} \label{sec:MoritaPhi-appendix}

In this appendix, we fill in some details for the proof of Theorem~\ref{thm:Moritapres}.

\begin{proposition} \label{prop:MoritaPhi-1}
We have that  
\[
\begin{array}{c}
(P, \lambda_{P}^{\Gamma(S)}, \;  \rho_{P}^{\Gamma(S')}) \in {}_{\Gamma(S)}\mathcal{T}_{\Gamma(S')}, \quad  \quad (Q, \lambda_{Q}^{\Gamma(S')}, \;  \rho_{Q}^{\Gamma(S)}) \in {}_{\Gamma(S')}\mathcal{T}_{\Gamma(S)}, \quad \text{ where}\\\\
\hspace{-.05in}{\small
\begin{array}{ll}
\medskip
\lambda_{P}^{\Gamma(S)}= \Gamma(\lambda_{\overline{P}}^S) \; \Gamma_{S,\overline{P}}: \Gamma(S) \otimes_{\mathcal{T}} P \to P, &
\rho_{P}^{\Gamma(S')}= \Gamma(\rho_{\overline{P}}^{S'}) \; \Gamma_{\overline{P},S'}:  P \otimes_{\mathcal{T}} \Gamma(S') \to P,\\
\lambda_{Q}^{\Gamma(S')}= \Gamma(\lambda_{\overline{Q}}^{S'}) \; \Gamma_{S',\overline{Q}}: \Gamma(S') \otimes_{\mathcal{T}} Q \to Q, &
\rho_{Q}^{\Gamma(S)}= \Gamma(\rho_{\overline{Q}}^{S}) \; \Gamma_{\overline{Q},S}:  Q \otimes_{\mathcal{T}} \Gamma(S) \to Q.
\end{array}
}
\end{array}
\]
\end{proposition}

\begin{proof}
It is straight-forward to check that $P$ is a right $\Gamma(S')$-module in $\mathcal{T}$ with action given by $\rho_P^{\Gamma(S')}$.
In a similar way, it can be seen that $P$ is a left $\Gamma(S)$-module in $\mathcal{T}$ with action $\lambda_P^{\Gamma(S)}$. 
Let us now check the left and right action compatibility for $P$. Consider the diagram, where  $\otimes:= \otimes_{\mathcal{S}}$ and we suppress the $\otimes_*$ symbols in morphisms below.
{\tiny
\[
\hspace{-.1in}
\xymatrix@C-1.5pc@R-.3pc{
(\Gamma(S)\otimes_{\mathcal{T}} P)\otimes_{\mathcal{T}}\Gamma(S')\ar[rrrr]^{\alpha_{\Gamma(S),P,\Gamma(S')}}\ar[ddd]_{\lambda_{P}^{\Gamma(S)} \id}\ar[dr]_{\Gamma_{S,\overline{P}} \id}&&&& \Gamma(S)\otimes_{\mathcal{T}} (P\otimes_{\mathcal{T}} \Gamma(S'))\ar[ddd]^{\id\; \rho_P^{\Gamma(S')}} \ar[ld]^{\id\; \Gamma_{\overline{P},S'}}\\
&\Gamma(S\otimes \overline{P})\otimes_{\mathcal{T}}\Gamma(S')\ar[ddl]_{\Gamma(\lambda_{\overline{P}}^S) \id \hspace{-.1in}}\ar[d]^{\Gamma_{S\otimes\overline{P},S'}} &(1)&\Gamma(S)\otimes_{\mathcal{T}}\Gamma(\overline{P}\otimes S')\ar[ddr]^{\hspace{-.1in}\id\;\Gamma(\rho_{\overline{P}}^{S'})} \ar[d]_{\Gamma_{S,\overline{P} \otimes S}}&\\
&\Gamma((S\otimes \overline{P})\otimes S')\ar[d]^{\Gamma(\lambda_{\overline{P}}^{S}\; \id)}_{(3)\hspace{.3in}}\ar[rr]_{\Gamma(\alpha_{S,\overline{P},S'})}&&\Gamma(S\otimes (\overline{P}\otimes S'))\ar[d]_{\Gamma(\id\;\rho_{\overline{P}}^{S'})}^{\hspace{.3in}(4)}&\\
P\otimes_{\mathcal{T}}\Gamma(S')\ar[drr]_{\rho_{P}^{\Gamma(S')}}\ar[r]^{\Gamma_{\overline{P},S'}}&\Gamma(\overline{P}\otimes S')\ar[dr]^(.55){\Gamma(\rho_{\overline{P}}^{S'})}&(2)&\Gamma(S\otimes \overline{P})\ar[dl]_(.55){\Gamma(\lambda_{\overline{P}}^S)}&\Gamma(S) \otimes_{\mathcal{T}} P \ar[dll]^{\lambda_{P}^{\Gamma(S)}} \ar[l]_{\Gamma_{S,\overline{P}}}\\
&&P&&
}
\]
}

\noindent Here, $(1)$ commutes as $\Gamma$ is a monoidal functor, and $(2)$ commutes since $\overline{P}~\in~ {}_S\C_{S'}$. The diagrams $(3)$ and $(4)$ commute due to the naturality of $\Gamma_{*,*}$,
and the triangles correspond to the definition of the left and right actions of $P$ in $ {}_{\Gamma(S)}\mathcal{T}_{\Gamma(S')}$.
Therefore, $(P, \lambda_{P}^{\Gamma(S)}, \;  \rho_{P}^{\Gamma(S')}) \in {}_{\Gamma(S)}\mathcal{T}_{\Gamma(S')} $.
Analogously, $(Q, \lambda_{Q}^{\Gamma(S')}, \;  \rho_{Q}^{\Gamma(S)}) \in {}_{\Gamma(S')}\mathcal{T}_{\Gamma(S)} $.
\end{proof}

\begin{proposition} \label{prop:MoritaPhi-2}
The epimorphisms
\[
\begin{array}{ll}
\medskip
\tau: P \otimes_{\Gamma(S')} Q \twoheadrightarrow \Gamma(S) &\in {}_{\Gamma(S)}\mathcal{T}_{\Gamma(S)},\\
\mu: Q \otimes_{\Gamma(S)} P \twoheadrightarrow \Gamma(S') &\in {}_{\Gamma(S')}\mathcal{T}_{\Gamma(S')},
\end{array}
\]
satisfy diagrams $(\ast)$ and $(\ast \ast)$ in Proposition~\ref{prop:Morita-bimod}(b).
\end{proposition}

\begin{proof}
Diagram $(\ast)$ corresponds to the following; $\otimes$ is understood from context:
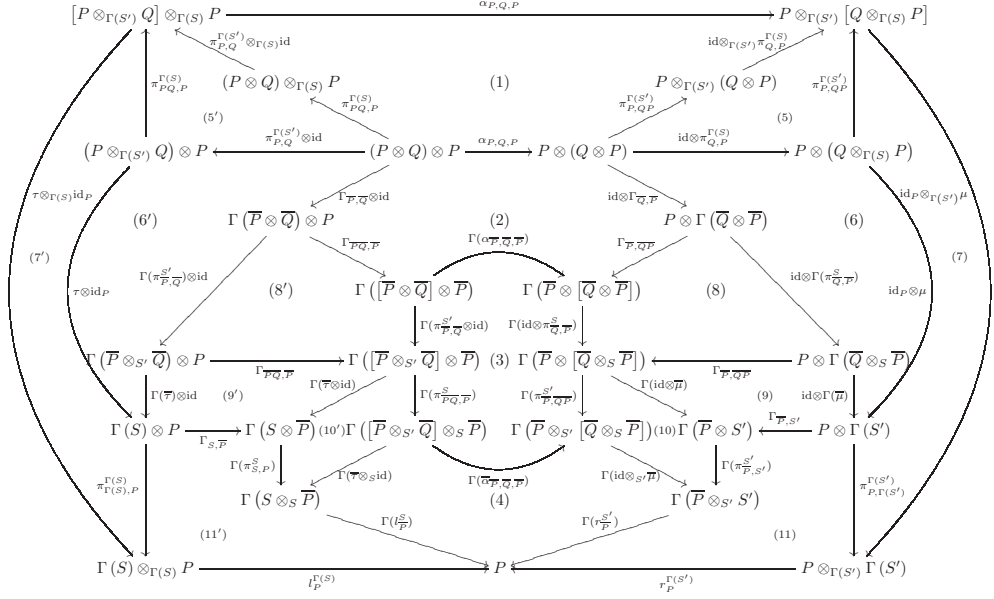
\begin{center}
\[
\resizebox{\displaywidth}{!}{
\xymatrix@C-2.5pc{
\left[P \otimes_{\Gamma(S')} Q \right] \otimes_{\Gamma(S)} P 
\ar@/_7pc/[dddddddd]^(.32){\tau \otimes_{\Gamma \left( S \right)} \id_P} 
\ar @{}[ddddddd]|{(7')\hspace{1.8in}} \ar[rrrrrr]^{\overline{\alpha}_{P,Q,P}} &&&&&& P \otimes_{\Gamma(S')} \left[ Q  \otimes_{\Gamma(S)} P \right]  
\ar@/^7pc/[dddddddd]_(.32){\id_P \otimes_{\Gamma \left( S' \right)} \mu}  
\ar @{}[ddddddd]|{\hspace{1.8in}(7)}
\\
& \left( P \otimes Q \right) \otimes_{\Gamma(S)} P \ar @{}[dl]|{(5')} \ar[ul]_(.4){\pi^{\Gamma(S')}_{P,Q}\otimes_{\Gamma(S)}\id} && {(1)} &&  P \otimes_{\Gamma(S')} \left( Q  \otimes P \right) \ar[ur]^(.4){\id\otimes_{\Gamma(S')}\pi^{\Gamma(S)}_{Q,P}} \ar @{}[dr]|{(5)}& \\
 \left( P \otimes_{\Gamma(S')} Q \right) \otimes P \ar[uu]_{\pi^{\Gamma(S)}_{PQ,P}} \ar@/_4pc/[dddd]^{\tau \otimes \id_P} && \left( P \otimes Q \right) \otimes P \ar[ul]_{\pi^{\Gamma(S)}_{PQ,P}} \ar[rr]^{\alpha_{P,Q,P}} \ar[dl]^{\Gamma_{\overline{P},\overline{Q}}\otimes\id} \ar[ll]_(.45){\pi^{\Gamma(S')}_{P,Q}\otimes\id} && P \otimes \left( Q \otimes P \right) \ar[rr]^(.45){\id\otimes \pi^{\Gamma(S)}_{Q,P}} \ar[ur]^{\pi^{\Gamma(S')}_{P,QP}} \ar[dr]_{\id\otimes \Gamma_{\overline{Q},\overline{P}}} && P \otimes \left( Q  \otimes_{\Gamma(S)} P \right) \ar[uu]^{\pi^{\Gamma(S')}_{P,QP}} \ar@/^4pc/[dddd]_{\id_P \otimes \mu}  \\
  (6') & \Gamma \left( \overline{P} \otimes \overline{Q} \right) \otimes P \ar[ddl]_{\Gamma(\pi^{S'}_{\overline{P},\overline{Q}})\otimes \id}\ar[dr]^{\Gamma_{\overline{PQ},\overline{P}}} && (2) && P \otimes \Gamma \left( \overline{Q} \otimes \overline{P} \right) \ar[ddr]^{\id\otimes\Gamma(\pi^S_{\overline{Q},\overline{P}})}  \ar[dl]_{\Gamma_{\overline{P},\overline{QP}}} & (6)  \\
 &(8')& \Gamma \left( \left[ \overline{P} \otimes \overline{Q} \right] \otimes \overline{P} \right) \ar@/^2pc/[rr]^{\Gamma(\alpha_{\overline{P},\overline{Q},\overline{P}})} \ar[d]^{\Gamma(\pi^{S'}_{\overline{P},\overline{Q}}\otimes\id)} && \Gamma \left( \overline{P} \otimes  \left[ \overline{Q} \otimes \overline{P} \right] \right) \ar[d]_{\Gamma(\id\otimes\pi^{S}_{\overline{Q},\overline{P}})} &(8)&   \\
 \Gamma \left( \overline{P} \otimes_{S'} \overline{Q} \right) \otimes P \ar[rr]_(.48){\Gamma_{\overline{PQ},\overline{P}}}\ar[d]^{\Gamma(\overline{\tau})\otimes\id} && \Gamma \left( \left[ \overline{P} \otimes_{S'} \overline{Q} \right] \otimes \overline{P} \right)  \ar[dl]_{\Gamma(\overline{\tau}\otimes\id)} \ar @{}[ddl]|{(10') \hspace{.27in}} \ar[d]^{\Gamma(\pi^{S}_{\overline{PQ},\overline{P}})} \ar @{}[dll]|{(9')\hspace{.8in}} & (3) & \Gamma \left( \overline{P} \otimes \left[ \overline{Q}  \otimes_S \overline{P} \right] \right) \ar @{}[ddr]|{\hspace{.27in}(10)} \ar @{}[drr]|{\hspace{.8in}(9)}  \ar[d]_{\Gamma(\pi^{S'}_{\overline{P},\overline{QP}})}  \ar[dr]^{\Gamma(\id\otimes\overline{\mu})} && P \otimes \Gamma \left( \overline{Q} \otimes_S \overline{P} \right) \ar[d]_{\id \otimes\Gamma(\overline{\mu})}\ar[ll]^(.45){\Gamma_{\overline{P},\overline{QP}}}  \\
\Gamma \left( S \right) \otimes P \ar[dd]_(.4){\pi^{\Gamma(S)}_{\Gamma(S),P}}\ar[r]_{\Gamma_{S,\overline{P}}}  & \Gamma \left( S \otimes \overline{P} \right)  \ar[d]_{\Gamma(\pi^{S}_{S,P})}  
&\; \Gamma \left( \left[ \overline{P} \otimes_{S'} \overline{Q} \right] \otimes_S \overline{P} \right) \ar@/_2pc/[rr]_{\Gamma(\overline{\alpha}_{\overline{P},\overline{Q},\overline{P}})} \ar[dl]^{\Gamma(\overline{\tau}\otimes_{S}\id)} && \Gamma \left( \overline{P} \otimes_{S'} \left[ \overline{Q} \otimes_S \overline{P} \right] \right) \ar[dr]_{\Gamma(\id\otimes_{S'}\overline{\mu})} & \Gamma \left( \overline{P} \otimes S' \right)
\ar[d]^{\Gamma(\pi^{S'}_{\overline{P},S'})}  & P \otimes \Gamma \left( S' \right) \ar[dd]^(.4){\pi^{\Gamma(S')}_{P,\Gamma(S')}} \ar[l]_{\Gamma_{\overline{P},S'}} \\
  & \Gamma \left( S \otimes_S \overline{P} \right) \ar@{}[dl]|{(11')}\ar[drr]^{\Gamma(l_{\overline{P}}^{S})} && (4) && \Gamma \left( \overline{P} \otimes_{S'} S' \right)\ar@{}[dr]|{(11)} \ar[dll]_{\Gamma(r_{\overline{P}}^{S'})} &   \\
\Gamma \left( S \right) \otimes_{\Gamma \left( S \right)} P
\ar[rrr]_{l_P^{\Gamma(S)}} &&& P &&& P \otimes_{\Gamma \left( S' \right)} \Gamma \left( S' \right) \ar[lll]^{r_P^{\Gamma(S')}}
\\
}
}
\]

\end{center}
Diagram~(1) is the definition of $\overline{\alpha}$ (see Definition~\ref{def:alphabar}). Diagram~(2) commutes as $\Gamma$ is a monoidal  functor, and~(3) results from applying $\Gamma$ to the definition of  $\overline{\alpha}$. Diagram~(4) is the result of applying the functor $\Gamma$ to the diagram (*). Diagrams~(5) and~(7) follow from~\eqref{eq:tenAmap}. Diagram~(6) is~\eqref{Tau hat - Mu hat}. Diagrams~(8) and~(9) commute from naturality of  $\Gamma_{*,*}$. Diagram~(10) commutes  by applying $\Gamma$ to~\eqref{eq:tenAmap}. The proof of diagram~(11) is given below.  Finally, the commutativity of~($5^{\prime}$)--($11^{\prime}$) follow analogously to the proof of~(5)--(11), respectively. Therefore, diagram ($\ast$) commutes. In an analogous manner, diagram ($\ast \ast$) commutes. 

\[
{\footnotesize
\xymatrix@C-2pc{
\Gamma(\overline{P}) \otimes_{\mathcal{T}} \Gamma(S') 
\ar[rr]^{\Gamma_{\overline{P},S'}}
\ar[d]_{\pi_{\Gamma(\overline{P}),\Gamma(S')}^{\Gamma(S')}}
\ar@/^1.7pc/[ddr]^{\rho_{\Gamma(\overline{P})}^{\Gamma(S')}}
&&
\Gamma(\overline{P} \otimes_{\mathcal{S}} S')
\ar[d]^{\Gamma(\pi_{\overline{P},S'}^{S'})}
\ar@/_1pc/[ddl]_(.4){\Gamma(\rho_{\overline{P}}^{S'})}\\
\Gamma(\overline{P}) \otimes_{\Gamma(S')} \Gamma(S')
\ar[dr]_{r_{\Gamma(\overline{P})}^{\Gamma(S')}}&
\text{\scriptsize{(def. of $r_{\Gamma(\overline{P})}^{\Gamma(S')}$) \; \;(def. of $\rho_P^{\Gamma(S')}$) \; \;  \; (def. of $r_{\overline{P}}^{S'}$)}}&
\Gamma(\overline{P} \otimes_{S'} S')
\ar[dl]^{\Gamma(r_{\overline{P}}^{S'})}\\
&\Gamma(\overline{P})&
}
}
\]
\end{proof}



\bibliography{WINART2}

\end{document}